\documentclass[a4paper, 10pt]{amsart}

\usepackage{amsmath, amsfonts, amssymb}
\usepackage{amsthm}
\usepackage{esint}
\usepackage[
bookmarks=false]{hyperref}
\usepackage{cite}
\usepackage{array}
\usepackage{amsaddr}
\usepackage{enumerate}
\title{\large{Pointwise estimates of weighted Bergman kernels in several complex variables}}

\author{Gian Maria Dall'Ara}
\address{Scuola Normale Superiore, Pisa, Italy}
\email{gianmaria.dallara@sns.it}

\date{\today}

\newcommand{\R}{\mathbb{R}}
\newcommand{\N}{\mathbb{N}}
\newcommand{\Q}{\mathbb{Q}}

\newcommand{\C}{\mathbb{C}}

\newcommand{\dbar}{\overline{\partial}}

\newcommand{\be}{\begin{equation*}}
\newcommand{\ee}{\end{equation*}}
\newcommand{\bel}{\begin{equation}}
\newcommand{\eel}{\end{equation}}
\newcommand{\bee}{\begin{eqnarray*}}
\newcommand{\eee}{\end{eqnarray*}}
\newcommand{\eps}{\varepsilon}

\newtheorem{thm}{Theorem}
\newtheorem{lem}[thm]{Lemma}
\newtheorem{prop}[thm]{Proposition}
\newtheorem{dfn}[thm]{Definition}

\begin{document}

\maketitle

\begin{abstract}
We prove new pointwise bounds for weighted Bergman kernels in $\C^n$, whenever a coercivity condition is satisfied by the associated weighted Kohn Laplacian on $(0,1)$-forms. Our results extend the ones obtained by Christ in \cite{christ}. 

Our main idea is to develop a version of Agmon theory (originally introduced in \cite{agmon} to deal with Schr\"odinger operators) for weighted Kohn Laplacians on $(0,1)$-forms, inspired by the fact that these are unitarily equivalent to certain generalized Schr\"odinger operators.
\end{abstract}

\section{Introduction}

\subsection{The problem and previous results}
The weighted Bergman kernel with respect to the weight $\varphi:\C^n\rightarrow\R$ is the integral kernel of the orthogonal projector of the weighted $L^2$ space\be
L^2(\C^n,\varphi):=\left\{f:\C^n\rightarrow\C\ :\ \int_{\C^n}|f(z)|^2e^{-2\varphi(z)}d\mathcal{L}(z)<+\infty\right\}
\ee onto its subspace $A^2(\C^n,\varphi)$ consisting of holomorphic functions (here and in the sequel $d\mathcal{L}(z)$ denotes Lebesgue measure in $\C^n$). See Section \ref{prel} for the precise definitions. The weighted Bergman kernel is a function \be
K_\varphi:\C^n\times\C^n\longrightarrow\C,
\ee
and the goal of this paper is to give pointwise estimates for its values under appropriate assumptions on $\varphi$. 

An extensive literature is devoted to the study of Bergman kernels in the general context of a complex manifold $M$ (when $M$ is not $\C^n$, the unweighted case $\varphi=0$ may already be very interesting), and a significant part of this literature has deep ties with harmonic analysis and partial differential equations (see, e.g., \cite{kerzman}, \cite{christ-regularity}, \cite{fefferman-kohn}, \cite{nagel-rosay-stein-wainger}, \cite{mcneal-stein-bergman}, \cite{koenig}, \cite{nagel-stein}).  

A particular motivation for the study of weighted Bergman kernels in $\C^n$ comes from the analysis of the Cauchy-Sz\"ego projection on certain model hypersurfaces in $\C^n$. In fact, if $\varphi$ is a plurisubharmonic non-harmonic polynomial, \be
M_\varphi:=\{z\in\C^{n+1}\ :\ \Im(z_{n+1})=\varphi(z_1,\dots,z_n)\}
\ee
is a good model for the boundary of a finite-type pseudo-convex domain (see \cite{dangelo-book}). The Cauchy-Sz\"ego projection is the orthogonal projector of $L^2(M_\varphi)$ onto the subspace of CR functions, where the unspecified measure is Lebesgue measure with respect to the coordinates $(z_1,\dots,z_n,\Re(z_{n+1}))$. The invariance under translation in the direction of $\Re(z_{n+1})$ allows to take a Fourier transform in that variable and reduce the study of the Cauchy-Sz\"ego projection to that of the family of weighted Bergman kernels $\{K_{\tau \varphi}\}_{\tau>0}$ (see \cite{has linger-bergmanszego}).

\par In the one-dimensional case, an important contribution to the analysis of $K_\varphi$ was given by Christ \cite{christ} (but see also \cite{marzo-ortegacerda}). Christ works under the assumption that $\varphi$ is subharmonic and that $\Delta\varphi(z)d\mathcal{L}(z)$ is a doubling measure giving a uniformly positive measure to euclidean discs of radius $1$. The hypotheses on $\Delta\varphi(z)d\mathcal{L}(z)$ are a sort of finite-type assumption, and are automatically verified when $\varphi$ is a subharmonic non-harmonic polynomial. 

The main result on $K_\varphi$ obtained by Christ is the estimate:
\be
|K_\varphi(z,w)|\leq Ce^{\varphi(z)+\varphi(w)}\frac{e^{-\eps d_0(z,w)}}{\rho_0(z)\rho_0(w)}\qquad\forall z,w\in\C,
\ee
where $C$ and $\eps$ are positive constants independent of $z$ and $w$,
\bel\label{intro-rho}
\rho_0(z):=\sup\left\{r>0:\ \int_{|w-z|\leq r}\Delta\varphi\leq 1\right\},
\eel
and $d_0$ is the Riemannian distance on $\C\equiv\R^2$ associated to the metric \be\rho_0(z)^{-2}(dx^2+dy^2).\ee 

Christ's proof appeals to the observation made by Berndtsson \cite{berndtsson} that the weighted Kohn Laplacian $\Box_\varphi$, an elliptic operator naturally occurring in this context (see Section \ref{prel}), is unitarily equivalent to a magnetic Schr\"odinger operator whose electrical potential is $\Delta\varphi$, and to Agmon theory, a powerful tool developed to establish exponential decay of eigenfunctions of Schr\"odinger operators with non-negative potentials (see \cite{agmon}). Thanks to the non-negativity of $\Delta\varphi$ and the diamagnetic inequality, one may ignore the magnetic potential when $n=1$ (see Section \ref{onevsmany-sec} for more details on this point).

The result of Christ was extended by Delin \cite{delin} to several complex variables under the assumption of strict plurisubharmonicity of the weight.

\subsection{Our results} Our goal in the present work is to prove pointwise estimates of $K_\varphi$ in several complex variables when the weight is (not necessarily strictly) plurisubharmonic. Trying to extend Christ's approach to $n\geq2$, one may observe that the weighted Kohn Laplacian is an operator acting on $(0,1)$-forms which is unitarily equivalent to a \emph{matrix magnetic Schr\"odinger operator}, a generalized magnetic Schr\"odinger operator whose electrical potential is Hermitian matrix-valued (see Section \ref{kohnschrod-sec} for the details), but here appears a serious difficulty: this electrical potential is never non-negative when $n\geq2$. Thus one cannot proceed in analogy with Christ's paper, via a diamagnetic inequality and an appropriate generalization of Agmon theory to the relevant class of generalized Schr\"odinger operators. 

Our way to overcome this obstacle is to apply the methods of Agmon theory directly to $\Box_\varphi$, without passing to Schr\"odinger operators. 

If $\mu:\C^n\rightarrow[0,+\infty]$, we say that $\Box_\varphi$ is \emph{$\mu$-coercive} if \be\Box_\varphi\geq\mu^2\ee as self-adjoint operators on the appropriate Hilbert spaces (here $\mu$ denotes the operator of multiplication by $\mu$). The precise definition is in Section \ref{coerc-sec}. We prove (Theorem \ref{bergman-thm}) that if $\varphi:\C^n\rightarrow\R$ is a plurisubharmonic weight such that $\Box_\varphi$ is $\kappa^{-1}$-coercive, and $\varphi$ and $\kappa$ meet some additional mild restrictions, then we have the following estimate for the weighted Bergman kernel:\bel\label{intro-mu}
|K_\varphi(z,w)|\leq Ce^{\varphi(z)+\varphi(w)}\frac{\kappa(z)}{\rho(z)}\frac{e^{-\eps d_\kappa(z,w)}}{\rho(z)^n\rho(w)^n}\qquad\forall z,w\in\C^n,
\eel
where \bel\label{intro-rho2}
\rho(z):=\sup\left\{r>0:\ \max_{|w-z|\leq r}\Delta\varphi(w)\leq r^{-2}\right\},
\eel
and $d_\kappa$ is the Riemannian distance associated to the metric \be\kappa(z)^{-2}\sum_{j=1}^n(dx_j^2+dy_j^2)\qquad(z_j=x_j+iy_j).\ee

Notice that \eqref{intro-rho} is better than \eqref{intro-rho2}, in the sense that, when $n=1$, $\rho\leq \sqrt{\pi}\rho_0$. This difference comes from the use of $L^\infty$ rather than $L^1$ bounds in our arguments. We do not consider this a serious limitation, since $\rho_0$ and $\rho$ are comparable when the weight $\varphi$ is a polynomial.

The important thing to observe is that the distance $d_0$ in Christ's estimate is replaced by $d_\kappa$ in our estimate, and that a factor $\frac{\kappa(z)}{\rho(z)}$ appears. 

If $\Box_\varphi$ is $c\rho^{-1}$-coercive (for some $c>0$), which is the case when the eigenvalues of the complex Hessian of $\varphi$ are comparable (Lemma \ref{mucomparable-lem}), then our result is the natural generalization of Christ's, that is
\be
|K_\varphi(z,w)|\leq Ce^{\varphi(z)+\varphi(w)}\frac{e^{-\eps d_\rho(z,w)}}{\rho(z)^n\rho(w)^n}\qquad\forall z,w\in\C^n.
\ee

In general, the best one should expect to be true is that $\Box_\varphi$ be $\kappa^{-1}$-coercive for some $\kappa$ larger than $\rho$. In this case the factor $\frac{\kappa}{\rho}$ reflects the non-comparability of eigenvalues. This phenomenon could be related to the appearance of an analogous term in the results of Nagel and Stein on decoupled domains (more precisely, the functions $B_k$ in Theorem 2.4.2 of \cite{nagel-stein}).

\subsection{Structure of the paper and a few details on the methods employed}

After introducing weighted Bergman spaces and kernels, $\dbar$-problems and Kohn Laplacians in Section \ref{prel}, in Section \ref{coerc-sec} we define the notion of $\mu$-coercivity for Kohn Laplacians. This is a very natural concept appearing (under different names) in a lot of literature on elliptic operators (from our perspective the most relevant example is \cite{agmon}), and we simply apply it to Kohn Laplacians. 

In Section \ref{radius-sec} we introduce another notion appearing in our main result, i.e., radius functions and associated distances. We also show how to associate a radius function to a potential. This is a known construction in the theory of Schr\"odinger operators (see \cite{shen}). 

Once all the ingredients are in place, in Section \ref{adm-sec} we define the class of \emph{admissible} weights to which our results apply. The next two sections are the heart of the paper. In Section \ref{exp-sec} we prove that whenever the weight is admissible and the weighted Kohn Laplacian is $\mu$-coercive for some $\mu$ satisfying a few mild hypotheses, the canonical solutions to $\dbar$-problems with certain compactly supported data exhibit an exponential decay which is quantitatively controlled by $\mu$. To deduce pointwise estimates of the weighted Bergman kernel from estimates of canonical solutions of the associated $\dbar$-problem, we use an argument sometimes dubbed \emph{Kerzman trick} which first appeared in \cite{kerzman}, and adapted to the weighted context by Delin \cite{delin}. This is done in Section \ref{bergman-sec}, where our main result (Theorem \ref{bergman-thm}) is stated and proved.

The rest of the paper is devoted to the specialization of Theorem \ref{bergman-thm} to admissible weights whose complex Hessian has comparable eigenvalues. Christ suggested in \cite{christ} that the discussion of this special case should be the first step to be taken in the study of weighted Bergman kernels in several variables. In order to do that, in Section \ref{kohnschrod-sec} we describe the unitary equivalence of weighted Kohn Laplacians and matrix Schr\"odinger operators, and apply in Section \ref{comp-sec} a version of the well-known Fefferman-Phong inequality, which we prove in Appendix \ref{fph-sec}. 

\subsection{Further directions} Theorem \ref{bergman-thm} is essentially a conditional result giving a non-trivial estimate for $K_\varphi(z,w)$ whenever one knows that $\Box_\varphi$ is $\mu$-coercive for some $\mu$. The larger the $\mu$, the better the estimate. We are thus naturally led to the problem of finding a $\mu$ such that $\Box_\varphi$ is $\mu$-coercive, when the eigenvalues of the complex Hessian of $\varphi$ are not comparable. In a forthcoming paper we will discuss this problem for certain classes of polynomial weights in $\C^2$ of the form 
\be\varphi_\Gamma(z_1,z_2):=\sum_{(\alpha_1,\alpha_2)\in\Gamma}|z_1|^{2\alpha_1} |z_2|^{2\alpha_2},\ee where $\Gamma$ is a finite subset of $\N^2$. Together with the present paper, this analysis will provide results that are somehow complementary to those obtained by Nagel and Pramanik \cite{nagel-pramanik-diagonal}, i.e., on-diagonal bounds for unweighted Bergman kernels on domains of the form $\Omega_\Gamma:=\{z\in\C^3\ :\ \Im(z_3)>\varphi_\Gamma(z_1,z_2)\}$ (and the analogous domains in $\C^{n+1}$, $n\geq3$). 

Another interesting aspect of this matter is the sharpness of the estimates. We speculate that proving optimal (or at least better) bounds for rather general weights may involve a generalization of the notion of $\mu$-coercivity, where $\mu$ is Hermitian matrix-valued ($\Box_\varphi$ acts on $(0,1)$-forms, or equivalently vectors, and hence it makes sense to multiply them pointwise by a matrix-valued $\mu$). This could take into account more precisely the vectorial nature of $\Box_\varphi$.

\subsection{Acknowledgements}

The present paper is part of the Ph.D. research conducted by the author at Scuola Normale Superiore in Pisa, under the supervision of Fulvio Ricci (see \cite{dallara-thesis}). The author would like to thank him for all the support, guidance and expertise offered during the last four years. The author is also very grateful to Alexander Nagel for inviting him at the university of Wisconsin, Madison, and giving him many helpful suggestions.

\section{Preliminaries}\label{prel}

In this section we introduce the most basic properties of the objects involved in our results: weighted Bergman spaces and kernels, weighted $\dbar$-problems and  Kohn Laplacians. 

We fix once and for all a \emph{weight} $\varphi:\C^n\longrightarrow\R$, which we assume to be $C^2$. Later more conditions will be imposed on it. 

\subsection{Weighted Bergman spaces and kernels}

We associate to $\varphi$ the weighted $L^2$ space $L^2(\C^n,\varphi)$ consisting of (equivalence classes of) functions $f:\C^n\rightarrow\C$ such that
\be
\int_{\C^n}|f(z)|^2e^{-2\varphi(z)}d\mathcal{L}(z)<+\infty.
\ee
We insert the factor $2$ in the exponential in order to slightly simplify several formulas later on. The Hilbert space norm and scalar product of $L^2(\C^n,\varphi)$ will be denoted by $||\cdot||_\varphi$ and $(\cdot,\cdot)_\varphi$.

The \emph{weighted Bergman space} with respect to the weight $\varphi$ is then defined as follows:\be
A^2(\C^n,\varphi):=\left\{ h:\C^n\rightarrow\C:\ h\text{ is holomorphic and } h\in L^2(\C^n,\varphi)\right\}.
\ee

If $h\in A^2(\C^n,\varphi)$, then in particular it is harmonic and satisfies the mean value property $h(z)=\frac{1}{|B(z,r)|}\int_{B(z,r)}h$. The Cauchy-Schwarz inequality yields\bel\label{berg-bound}
|h(z)|\leq \frac{1}{|B(z,r)|}\sqrt{\int_{B(z,r)}e^{2\varphi}}\ ||h||_\varphi\qquad\forall h\in A^2(\C^n,\varphi),
\eel for any $z\in\C^n$ and $r>0$. This estimate has two elementary consequences:\begin{itemize}
\item[(a)] $A^2(\C^n,\varphi)$ is a closed subspace of $L^2(\C^n,\varphi)$ (by \eqref{berg-bound} convergence of a sequence of $A^2(\C^n,\varphi)$ in the $||\cdot||_\varphi$-norm implies uniform convergence, which preserves holomorphicity). We denote by $B_\varphi$ the orthogonal projector of $L^2(\C^n,\varphi)$ onto $A^2(\C^n,\varphi)$.
\item[(b)] The evaluation mappings $h\mapsto h(z)$ are continuous linear functionals of $A^2(\C^n,\varphi)$, and Riesz Lemma yields a function $K_\varphi:\C^n\times \C^n\rightarrow\C$ such that \bel\label{berg-eq}
h(z)=\int_{\C^n}K_\varphi(z,w)h(w)e^{-2\varphi(w)}d\mathcal{L}(w)\qquad\forall z\in\C^n,
\eel and $\overline{K_\varphi(z,\cdot)}\in A^2(\C^n,\varphi)$ for every $z\in\C^n$. 
\end{itemize}
The operator $B_\varphi$ is called the \emph{weighted Bergman projector} and the function $K_\varphi$ the \emph{weighted Bergman kernel} associated to the weight $\varphi$. It is immediate to see that \be
B_\varphi(f)(z)=\int_{\C^n}K_\varphi(z,w)f(w)e^{-2\varphi(w)}d\mathcal{L}(w)\qquad\forall z\in\C^n,
\ee for every $f\in L^2(\C^n,\varphi)$, i.e., $K_\varphi$ is the integral kernel of $B_\varphi$. Since $B_\varphi$ is self-adjoint, $K_\varphi(z,w)=\overline{K_\varphi(w,z)}$. In particular $K_\varphi(\cdot,w)\in A^2(\C^n,\varphi)$ for every $w\in\C^n$, and \eqref{berg-eq} gives\bel\label{berg-diag}
K_\varphi(z,z)=\int_{\C^n}|K(z,w)|^2e^{-2\varphi(w)}d\mathcal{L}(w)=\sup_{h\in A^2(\C^n,\varphi),\ ||h||_\varphi=1} |h(z)|^2, \eel where the last term is the operator norm squared of the evaluation functional. Identity \eqref{berg-diag} is the main route to pointwise estimates of the diagonal values of the weighted Bergman kernel. Unfortunately there is not an equally neat variational characterization of non-diagonal values.

\subsection{Weighted $\dbar$-problems}

We begin by recalling the classical formalism of the $\dbar$ complex. We denote by $L^2_{(0,q)}(\C^n,\varphi)$ the Hilbert space of $(0,q)$-forms with coefficients in $L^2(\C^n,\varphi)$.   
Since we will be working only with forms of degree less than or equal to $2$, we confine our discussion to these cases. Adopting the standard notation for differential forms, we have that $L^2_{(0,0)}(\C^n,\varphi)=L^2(\C^n,\varphi)$,\be
L^2_{(0,1)}(\C^n,\varphi):=\left\{u=\sum_{1\leq j\leq n} u_j d\overline{z}_j:\ u_j\in L^2(\C^n,\varphi)\quad\forall j\right\},
\ee and
\be
L^2_{(0,2)}(\C^n,\varphi):=\left\{w=\sum_{1\leq j<k\leq n} w_{jk}\ d\overline{z}_j\wedge d\overline{z}_k:\ w_{jk}\in L^2(\C^n,\varphi)\quad\forall j,k\right\}.
\ee
For the norms and the scalar products in these Hilbert spaces of forms we use the same symbols $||\cdot||_\varphi$ and $(\cdot,\cdot)_\varphi$, i.e., if $u,\widetilde{u}\in L^2_{(0,1)}(\C^n,\varphi)$, we have\be
||u||_\varphi^2=\sum_{1\leq j\leq n}||u_j||_\varphi^2,\quad (u,\widetilde{u})_\varphi=\sum_{1\leq j\leq n}(u_j,\widetilde{u}_j)_\varphi,
\ee while if $w,\widetilde{w}\in L^2_{(0,2)}(\C^n,\varphi)$, we have\be
||w||_\varphi^2=\sum_{1\leq j<k\leq n}||w_{jk}||_\varphi^2,\quad (w,\widetilde{w})_\varphi=\sum_{1\leq j<k\leq n}(w_{jk},\widetilde{w}_{jk})_\varphi.
\ee 
The meaning of $||\cdot||_\varphi$ and $(\cdot,\cdot)_\varphi$ depends on whether the arguments are functions, $(0,1)$-forms or $(0,2)$-forms, but this ambiguity should not be a source of confusion. Observe that the formulas above reveal the nature of product Hilbert space of $L^2_{(0,q)}(\C^n,\varphi)$. 

We now introduce the initial fragment of the \emph{weighted $\dbar$-complex}:
\bel\label{dbar-complex}
L^2(\C^n,\varphi)\stackrel{\dbar}\longrightarrow L^2_{(0,1)}(\C^n,\varphi)\stackrel{\dbar}\longrightarrow L^2_{(0,2)}(\C^n,\varphi).
\eel
The symbol $\dbar$ denotes as usual both the operator $\dbar:L^2(\C^n,\varphi)\rightarrow L^2_{(0,1)}(\C^n,\varphi)$ defined on the domain \be
\mathcal{D}_0(\dbar):=\left\{f\in L^2(\C^n,\varphi):\ \frac{\partial f}{\partial \overline{z}_j}\in L^2(\C^n,\varphi)\ \forall j\right\}
\ee by the formula $\dbar f=\sum_j\frac{\partial f}{\partial \overline{z}_j}d\overline{z}_j$, and the operator $\dbar: L^2_{(0,1)}(\C^n,\varphi)\rightarrow L^2_{(0,2)}(\C^n,\varphi)$ defined on the domain \be
\mathcal{D}_1(\dbar):=\left\{u=\sum_ju_jd\overline{z}_j\in L^2_{(0,1)}(\C^n,\varphi):\ \frac{\partial u_k}{\partial \overline{z}_j}-\frac{\partial u_j}{\partial \overline{z}_k}\in L^2(\C^n,\varphi)\ \forall j,k\right\}
\ee by the formula $\dbar u=\sum_{j<k}\left(\frac{\partial u_k}{\partial \overline{z}_j}-\frac{\partial u_j}{\partial \overline{z}_k}\right)d\overline{z}_j\wedge d\overline{z}_k$. 

The weighted $\dbar$-complex \eqref{dbar-complex} is a complex, i.e., \bel\label{dbar-squared}
\dbar f\in \mathcal{D}_1(\dbar) \quad \text{and}\quad \dbar\dbar f=0\qquad\forall f\in \mathcal{D}_0(\dbar),
\eel and the kernel of $\dbar:L^2(\C^n,\varphi)\rightarrow L^2_{(0,1)}(\C^n,\varphi)$ is of course $A^2(\C^n,\varphi)$.

Therefore, if $u\in L^2_{(0,1)}(\C^n,\varphi)$ is $\dbar$-closed, i.e., $\dbar u=0$, the \emph{weighted $\dbar$-problem}\bel\label{dbar-prob}
\dbar f=u
\eel admits at most one solution $f\in \mathcal{D}_0(\dbar)$ orthogonal to $A^2(\C^n,\varphi)$. In case it exists, it is called the \emph{canonical solution} to \eqref{dbar-prob}.

\subsection{Weighted Kohn Laplacians}

Taking the Hilbert space adjoints of the operators in \eqref{dbar-complex} (as we can, since the operators are closed and densely defined), we have the dual complex:
\bel\label{kohn-dual}
L^2(\C^n,\varphi)\stackrel{\dbar^*_\varphi}\longleftarrow L^2_{(0,1)}(\C^n,\varphi)\stackrel{\dbar^*_\varphi}\longleftarrow L^2_{(0,2)}(\C^n,\varphi).
\eel

We use the index $\varphi$ in the symbols for these operators to stress the fact that not only the domain, but also the formal expression of $\dbar^*_\varphi$ depends on the weight $\varphi$. In particular, on $(0,1)$-forms the expression is\bel\label{dbarstar}
\dbar^*_\varphi \left(\sum_{j=1}^nu_jd\overline{z}_j\right)=\sum_{j=1}^n\left(-\partial_ju_j+\partial_j\varphi\cdot u_j\right).
\eel

The \emph{weighted Kohn Laplacian} is defined by the formula
\be
\Box_\varphi:=\dbar^*_\varphi\dbar+\dbar\dbar^*_\varphi
\ee 
on the domain of $(0,1)$-forms \be
\mathcal{D}(\Box_\varphi):=\{u\in L^2_{(0,1)}(\C^n,\varphi):\ u\in\mathcal{D}_1(\dbar)\cap\mathcal{D}_1(\dbar^*_\varphi),\ \dbar u\in \mathcal{D}_2(\dbar^*_\varphi)\text{ and }\dbar^* _\varphi u\in \mathcal{D}_0(\dbar)\}.
\ee 

The weighted Kohn Laplacian is a densely-defined, closed, self-adjoint and non-negative operator on $L^2_{(0,1)}(\C^n,\varphi)$. The details of the routine arguments proving this fact can be found in \cite{haslinger-book} (or in \cite{chen-shaw} for the very similar unweighted case).

Finally, let us introduce the quadratic form \be
\mathcal{E}_\varphi(u,v):=(\dbar u,\dbar v)_\varphi+(\dbar^*_\varphi u,\dbar^*_\varphi v)_\varphi,
\ee defined for $u,v\in\mathcal{D}(\mathcal{E}_\varphi):=\mathcal{D}_1(\dbar)\cap\mathcal{D}_1(\dbar^*_\varphi)$. Notice that, by definition of Hilbert space adjoints, \be
(\Box_\varphi u,v)=\mathcal{E}_\varphi(u,v) \qquad \forall u\in \mathcal{D}(\Box_\varphi),\quad \forall v\in \mathcal{D}(\mathcal{E}_\varphi).
\ee
We will simply write $\mathcal{E}_\varphi(u)$ for $\mathcal{E}_\varphi(u,u)$. The well-known Morrey-Kohn-H\"ormander formula gives an alternative expression for $\mathcal{E}_\varphi(u)$. In order to state it, we define the \emph{complex Hessian} $H_\varphi=\left(\frac{\partial^2\varphi}{\partial z_j\partial \overline{z}_k}\right)_{j,k=1}^n$, a continuous mapping defined on $\C^n$ and whose values are $n\times n$ Hermitian matrices. We also identify the $(0,1)$-form $u=\sum_{j=1}^n u_jd\overline{z}_j$ with the vector field $u=(u_1,\dots,u_n):\C^n\rightarrow\C^n$, so that $(H_\varphi u,u)=\sum_{j,k=1}^n\frac{\partial^2\varphi}{\partial z_j\partial \overline{z}_k}u_j\overline{u}_k$. 
The Morrey-Kohn-H\"ormander formula is the following identity:
\bel\label{MKH}
\mathcal{E}_\varphi(u)=\sum_{j,k}\int_{\C^n}|\dbar_k u_j|^2e^{-2\varphi}+2\int_{\C^n} (H_\varphi u,u)e^{-2\varphi}\qquad\forall u\in\mathcal{D}(\mathcal{E}_\varphi).
\eel 
A proof may be found in \cite{haslinger-book} or \cite{chen-shaw}. Identity \eqref{MKH} reveals the fundamental role played by $H_\varphi$ in the analysis of $\Box_\varphi$. In particular, in view of \eqref{MKH} it is very natural and useful to assume that the weight $\varphi$ be \emph{plurisubharmonic}, i.e., 
\bel\label{plush}
(H_\varphi(z)v,v)\geq0\qquad\forall z\in\C^n,\quad v\in\C^n.
\eel

\subsection{A useful identity and a Caccioppoli-type inequality involving $\mathcal{E}_\varphi$}

\begin{prop}\label{MKH-computation}
Assume that $u\in \mathcal{D}(\Box_\varphi)$ and let $\eta$ be a real-valued bounded Lipschitz function. Then $\eta u\in \mathcal{D}(\mathcal{E}_\varphi)$ and 
\be 
\mathcal{E}_\varphi(\eta u)=\frac{1}{4}\int_{\C^n}|\nabla \eta|^2|u|^2e^{-2\varphi} + \Re(\eta\Box_\varphi u,\eta u)_\varphi.
\ee
\end{prop}

\begin{proof}
We omit the easy verification that $\eta u\in \mathcal{D}(\mathcal{E}_\varphi)$. Then we have: \bee
|\dbar_k(\eta u_j)|^2&=&|\dbar_k\eta|^2 |u_j|^2+\Re\left(\eta^2|\dbar_ku_j|^2+2\eta\partial_k\eta \cdot\overline{u_j}\dbar_k u_j\right)\\
&=&|\dbar_k\eta|^2 |u_j|^2+\Re\left(\dbar_k u_j(\eta^2\partial_k\overline{u_j}+\partial_k(\eta^2) \overline{u_j})\right)\\
&=&|\dbar_k\eta|^2  |u_j|^2+\Re\left(\dbar_k u_j\overline{\dbar_k(\eta^2u_j)}\right).
\eee 
Integrating this identity and using the polarized version of the Morrey-Kohn-H\"ormander formula we obtain\bee
\mathcal{E}_\varphi(\eta u)&=&\frac{1}{4}\int_{\C^n}|\nabla \eta|^2|u|^2e^{-2\varphi}+\Re\left(\sum_{j,k=1}^n\int_{\C^n}\dbar_k u_j\overline{\dbar_k(\eta^2u_j)}e^{-2\varphi}\right) \\
&+& 2\Re\left(\int_{\C^n}(H_\varphi u,\eta^2 u)e^{-2\varphi}\right)\\
&=&\frac{1}{4}\int_{\C^n}|\nabla \eta|^2|u|^2e^{-2\varphi}+\Re\left(\mathcal{E}_\varphi(u,\eta^2u)\right).
\eee Since $\eta^2u\in \mathcal{D}(\mathcal{E}_\varphi)$ and $u\in \mathcal{D}(\Box_\varphi)$, we have $\mathcal{E}_\varphi(u,\eta^2u)=(\Box_\varphi u,\eta^2u)_\varphi=(\eta\Box_\varphi u,\eta u)_\varphi$. This concludes the proof. 
\end{proof}

We now state and prove the Caccioppoli-type inequality.

\begin{lem}\label{MKH-caccioppoli}
Assume that $u\in \mathcal{D}(\Box_\varphi)$ and that $\Box_\varphi u$ vanishes on $B(z,R)$. If $r<R$, then\be
\int_{B(z,r)}|\dbar^*_\varphi u|^2e^{-2\varphi}\leq (R-r)^{-2} \int_{B(z,R)}|u|^2e^{-2\varphi}.
\ee 
\end{lem}

\begin{proof} Let $\eta$ be Lipschitz, real-valued, identically equal to $1$ on $B(z,r)$, and supported on $B(z,R)$. Since $\eta\Box_\varphi u=0$, Proposition \ref{MKH-computation} yields\be
||\dbar^*_\varphi (\eta u)||_\varphi^2\leq \mathcal{E}_\varphi(\eta u)=\frac{1}{4}\int_{\C^n}|\nabla \eta|^2|u|^2e^{-2\varphi}\leq \frac{||\nabla \eta||_\infty^2}{4}||\chi_{B(z,R)} u||_\varphi^2.
\ee
Since $\dbar^*_\varphi(\eta u)=\eta\dbar^*_\varphi u - \sum_{j=1}^n\partial_j\eta\cdot u_j$ (recall \eqref{dbarstar}), we have\bee
||\chi_{B(z,r)} \dbar^*_\varphi u||_\varphi&\leq& ||\eta\dbar^*_\varphi u||_\varphi\\
&\leq& ||\dbar^*_\varphi (\eta u)||_\varphi+\frac{||\nabla\eta||_\infty}{2}||\chi_{B(z,R)} u||_\varphi\\
&\leq &||\nabla\eta||_\infty||\chi_{B(z,R)} u||_\varphi.
\eee It is clear that we can choose $\eta$ such that $||\nabla\eta||_\infty=\frac{1}{R-r}$, and this gives the thesis.\end{proof}

\section{$\mu$-coercivity for weighted Kohn Laplacians}\label{coerc-sec}

\begin{dfn}\label{coerc-def}
Given a measurable function $\mu:\C^n\rightarrow[0,+\infty)$, we say that $\Box_\varphi$ is \emph{$\mu$-coercive} if the following inequality holds
\bel\label{coerc-formula}
\mathcal{E}_\varphi(u)\geq ||\mu u||_\varphi^2 \qquad\forall u\in \mathcal{D}(\mathcal{E}_\varphi).
\eel
\end{dfn}

The next proposition collects a few basic facts about $\mu$-coercivity. 

\begin{prop}\label{coerc-prop} Assume that $\Box_\varphi$ is $\mu$-coercive, and that 
\bel\label{coerc-inf}\inf_{z\in\C^n}\mu(z)>0.\eel
Then:
\begin{enumerate}
\item[\emph{(i)}] $\Box_\varphi$ has a bounded, self-adjoint and non-negative inverse $N_\varphi$ such that \bel\label{coerc-neumann-bound}
||\mu N_\varphi g||_\varphi\leq ||\mu^{-1} g||_\varphi \qquad\forall g\in L^2_{(0,1)}(\C^n,\varphi).
\eel
\item[\emph{(ii)}] The weighted $\dbar$ equation is solvable, and $\dbar^*_\varphi N_\varphi$ is the canonical solution operator of the weighted $\dbar$-problem, i.e., if $u\in L^2_{(0,1)}(\C^n,\varphi)$ is $\dbar$-closed, then $f:=\dbar^*_\varphi N_\varphi u$ is such that $\dbar f=u$ and $f$ is orthogonal to $A^2(\C^n,\varphi)$. Moreover, \bel\label{coerc-canonical-bound}
||f||_\varphi\leq ||\mu^{-1} u||_\varphi.
\eel 
\item[\emph{(iii)}] We have the identity
\bel\label{coerc-bergman} B_\varphi f= f-\dbar^*_\varphi N_\varphi \dbar f\qquad \forall f\in \mathcal{D}_0(\dbar).\eel
\end{enumerate}
\end{prop}

The operator $N_\varphi$ is customarily called the \emph{$\dbar$-Neumann operator}. 

\begin{proof}

(i) To see that $\Box_\varphi$ is injective, observe that, if $\Box_\varphi u=0$, inequality \eqref{coerc-formula} implies that $mu=0$ and hence, by \eqref{coerc-inf}, that $u=0$. By self-adjointness, $\Box_\varphi$ has dense range. If $g\in L^2_{(0,1)}(\C^n,\varphi)$, the anti-linear functional $\lambda_g:\Box_\varphi u\mapsto (g,u)_\varphi$ is then well-defined on a dense subset of $L^2_{(0,1)}(\C^n,\varphi)$. It satisfies the bound\bel\label{coerc-lambda}
|\lambda_g(\Box_\varphi u)|=|(g,u)_\varphi|\leq ||\mu^{-1} g||_\varphi||\mu u||_\varphi.
\eel
An application of the Cauchy-Schwarz inequality shows that $\mu$-coercivity implies, for any $u\in\mathcal{D}(\Box_\varphi)$, \be
||\mu u||_\varphi^2\leq \mathcal{E}_\varphi(u)=(\Box_\varphi u,u)_\varphi\leq ||\mu^{-1}\Box_\varphi u||_\varphi||\mu u||_\varphi,
\ee i.e., $||\mu u||_\varphi\leq||\mu^{-1}\Box_\varphi u||_\varphi$ ($||\mu u||_\varphi$ is finite for any $u\in\mathcal{D}(\Box_\varphi)$ by $\mu$-coercivity).
Plugging this inequality into \eqref{coerc-lambda}, we obtain \bel\label{coerc-existence}
|\lambda_g(\Box_\varphi u)|\leq ||\mu^{-1} g||_\varphi||\mu^{-1}\Box_\varphi u||_\varphi .
\eel
Since $\mu^{-1}$ is bounded, $\lambda_g$ may be uniquely extended to a continuous anti-linear functional on $L^2_{(0,1)}(\C^n,\varphi)$ and hence there exists $N_\varphi g\in L^2_{(0,1)}(\C^n,\varphi)$ such that $(u,g)_\varphi=(\Box_\varphi u,N_\varphi g)_\varphi$ for every $u\in \mathcal{D}(\Box_\varphi)$. This means that $N_\varphi g\in\mathcal{D}(\Box_\varphi)$ and that $\Box_\varphi N_\varphi g=g$. In particular $\Box_\varphi$ is surjective and $N_\varphi$, being the inverse of $\Box_\varphi$, is a bounded, self-adjoint, and non-negative operator. Inequality \eqref{coerc-neumann-bound} follows from \eqref{coerc-existence}:\be
||\mu N_\varphi g||_\varphi=\sup_{||w||_\varphi=1}|(\mu N_\varphi g,w)_\varphi|=\sup_{||w||_\varphi=1,||\mu w||_\varphi<+\infty}|\lambda_g(\mu w)|\leq ||\mu^{-1} g||_\varphi.
\ee In the second identity we used the fact that $\mu w\in L^2_{(0,1)}(\C^n,\varphi)$ for $w$ in a dense subspace as a consequence of $\mu$-coercivity.

(ii) Let $u$ and $f$ be as in the statement. We compute\be
\dbar f = \dbar\dbar^*_\varphi N_\varphi u= \Box_\varphi N_\varphi u-\dbar^*_\varphi \dbar N_\varphi u=u-\dbar^*_\varphi \dbar N_\varphi u.
\ee Notice that $N_\varphi u\in \mathcal{D}(\Box_\varphi)$ and hence the computation is meaningful. Since $u$ and $\dbar f$ are both $\dbar$-closed, the identity above implies that $\dbar^*_\varphi \dbar N_\varphi u\in \mathcal{D}_1(\dbar)$ and $\dbar\dbar^*_\varphi \dbar N_\varphi u=0$. In particular \be
0=(\dbar\dbar^*_\varphi \dbar N_\varphi u,\dbar N_\varphi u)_\varphi=||\dbar^*_\varphi \dbar N_\varphi u||_\varphi^2
\ee and hence $\dbar f=u$. Notice that one can repeat the above argument to show that $\dbar N_\varphi u=0$, but we don't need this fact. The solution $f$ is obviously orthogonal to $A^2(\C^n,\varphi)$, because it is in the range of $\dbar^*_\varphi$. To obtain the bound on $f$ we observe that\bee
||f||_\varphi^2&=&(\dbar^*_\varphi N_\varphi u,f)_\varphi=(N_\varphi u,\dbar f)_\varphi\\
&=&(N_\varphi u,u)_\varphi=(\mu N_\varphi u,\mu^{-1}u)_\varphi\leq ||\mu^{-1} u||_\varphi^2,
\eee where the last inequality is \eqref{coerc-neumann-bound}.

(iii) Let $f\in\mathcal{D}_0(\dbar)$. Since $\dbar f$ is $\dbar$-closed, part \emph{(ii)} shows that $\dbar^*_\varphi N_\varphi \dbar f$ is orthogonal to $A^2(\C^n,\varphi)$ and that $f-\dbar^*_\varphi N_\varphi \dbar f\in A^2(\C^n,\varphi)$. Hence $B_\varphi f=f-\dbar^*_\varphi N_\varphi \dbar f$.\end{proof}

\section{Radius functions and associated distances}\label{radius-sec}

\subsection{Definitions and basic properties}

We say that $\rho:\R^d\rightarrow (0,+\infty)$ is a \emph{radius function} if it is Borel and there exists a constant $C<+\infty$ such that for every $x\in\R^d$ we have\bel\label{radius-ineq}
C^{-1}\rho(x)\leq \rho(y)\leq C\rho(x) \qquad \forall y\in B(x,\rho(x)).
\eel

In other words, a radius function $\rho$ is approximately constant on the ball centered at $x$ of radius $\rho(x)$. 

To any radius function $\rho$, we associate the Riemannian metric $\rho(x)^{-2}dx^2$. In fact, we are interested only in the associated Riemannian distance, which we describe explicitly. If $I$ is a compact interval and $\gamma:I\rightarrow\R^d$ is a piecewise $C^1$ curve, we define\be
L_\rho(\gamma):=\int_I\frac{|\gamma'(t)|}{\rho(\gamma(t))}dt.
\ee 
Notice that the integrand $\frac{|\gamma'(t)|}{\rho(\gamma(t))}$ is defined on the complement of the finite set of times where $\gamma'$ is discontinuous, and it is a measurable function, because $\rho$ is assumed to be Borel. Moreover, the integral is absolutely convergent because $\rho^{-1}$ is locally bounded. 

Given $x,y\in\R^d$, we put\be
d_\rho(x,y):=\inf_\gamma L_\rho(\gamma),
\ee where the $\inf$ is taken as $\gamma$ varies over the collection of curves connecting $x$ and $y$. Finally, we define $B_\rho(x,r):=\{y\in\R^d:\ d_\rho(x,y)<r\}$. 

\begin{prop}\label{radius-prop} \par The function $d_\rho$ just defined is a distance and \be
B_\rho(x,C^{-1}r)\subseteq B(x,r\rho(x))\subseteq B_\rho(x,Cr) \qquad\forall r\leq1,\ x\in\R^d,
\ee where $C$ is the constant appearing in \eqref{radius-ineq}. Moreover, the function \be
y\mapsto d_\rho(x,y)\ee 
is locally Lipschitz for every $x$, and for almost every $y\in \R^d$ we have \bel\label{radius-nabla}
|\nabla_y d_\rho(x,y)|\leq \frac{C}{\rho(y)}.
\eel
\end{prop}

\begin{proof}
Let $x,y\in\R^d$ be such that $|x-y|=s\rho(x)$, for some $s>0$. Take any piecewise $C^1$ curve $\gamma:[0,T]\rightarrow\R^d$ connecting $x$ and $y$, and let $T_0$ be the minimum time such that $|x-\gamma(T_0)|=\min\{s,1\}\rho(x)$. By \eqref{radius-ineq}, $\rho(\gamma(t))\leq C\rho(x)$ for every $t\in[0,T_0)$. Hence, \bee
L_\rho(\gamma)&\geq& \int_0^{T_0}\frac{|\gamma'(t)|}{\rho(\gamma(t))}dt\geq \frac{C^{-1}}{\rho(x)}\int_0^{T_0}|\gamma'(t)|dt\\
&\geq& \frac{C^{-1}}{\rho(x)} \min\{s,1\}\rho(x)=C^{-1}\min\{s,1\}.
\eee By the arbitrariness of $\gamma$, we conclude that $d_\rho(x,y)\geq C^{-1}\min\left\{\frac{|x-y|}{\rho(x)},1\right\}$. This implies that $d_\rho$ is non-degenerate and hence a genuine distance, triangle inequality and symmetry being obvious. 
It also shows that if $y$ lies in $B_\rho(x,C^{-1}r)$ ($r\leq1$), then $r>\min\left\{\frac{|x-y|}{\rho(x)},1\right\}$, and then the minimum has to be equal to $\frac{|x-y|}{\rho(x)}$, proving the first inclusion of the statement.

To prove the second inclusion, we use the fact that $\rho(u)\geq C^{-1}\rho(x)$ for every $u\in B(x,r\rho(x))$, if $r\leq 1$. Given $y\in B(x,r\rho(x))$, define $\sigma(t)=x+t(y-x)$ and notice that \be
d_\rho(x,y)\leq \int_0^1\frac{|\sigma'(t)|}{\rho(\sigma(t))}dt\leq \frac{C}{\rho(x)}|x-y|< Cr,
\ee so that $B(x,r\rho(x))\subseteq B_\rho(x,Cr)$, that is the second inclusion to be proved.

Fix now $x,y\in \R^d$ and let $h\in\R^d$ be such that $|h|<\rho(y)$. As above, we have $d_\rho(y,y+h)\leq C\frac{|h|}{\rho(y)}$. The triangle inequality then yields\bel\label{radius-lip}
|d_\rho(x,y+h)-d_\rho(x,y)|\leq d_\rho(y,y+h)\leq C\frac{|h|}{\rho(y)}\qquad\forall h: |h|<\rho(y).
\eel By the local boundedness of $\rho^{-1}$, we conclude that $d_\rho(x,\cdot)$ is locally Lipschitz. Rademacher's theorem implies that $d_\rho$ is almost everywhere differentiable and \eqref{radius-lip} translates into \eqref{radius-nabla}.\end{proof}

We conclude this section with two elementary propositions. The second one is a very classical construction of a covering.

\begin{prop}\label{radius-max}
If $\rho_1$ and $\rho_2$ are two radius functions on $\R^d$, then $\rho_1\vee\rho_2:=\max\{\rho_1,\rho_2\}$ is a radius function.
\end{prop}

\begin{proof}
Assume that $C<+\infty$ is a constant for which \eqref{radius-ineq} holds both for $\rho_1$ and $\rho_2$.

Fix $x\in\R^d$ and assume that $\rho_1\vee\rho_2(x)=\rho_1(x)$. If $y\in B(x,\rho_1\vee\rho_2(x))$, then the first inequality in \eqref{radius-ineq} for $\rho_1$ yields \be
\rho_1\vee\rho_2(x)=\rho_1(x)\leq C\rho_1(y)\leq C\rho_1\vee\rho_2(y).
\ee
If $\rho_1\vee\rho_2(x)=\rho_2(x)$ the conclusion would be the same (using \eqref{radius-ineq} for $\rho_2$).

Now there are two possibilities: either $\rho_1\vee\rho_2(x)\leq \rho_1\vee\rho_2(y)$, in which case the same argument with $x$ and $y$ swapped gives the bound $\rho_1\vee\rho_2(y)\leq C\rho_1\vee\rho_2(x)$, or the converse inequality $\rho_1\vee\rho_2(y)< \rho_1\vee\rho_2(x)$ holds. In both cases the proof is completed.
\end{proof}

\begin{prop}\label{radius-covering}
If $\rho$ is a radius function there is a countable set $\{x_k\}_{k\in\N}\subseteq \R^d$ such that: \begin{enumerate}
\item[\emph{(i)}] $\left\{B(x_k,\rho(x_k))\right\}_{k\in\N}$ is a covering of $\R^d$,
\item[\emph{(ii)}] any $x\in\R^d$ lies in at most $K$ of the balls of the covering, where $K$ depends only on $C$ and $d$.
\end{enumerate}
\end{prop}

\begin{proof}
Let $\{x_k\}_{k\in\N}$ be such that $\mathcal{B}=\left\{B(x_j,\frac{\rho(x_j))}{1+C^2}\right\}_{j\in\N}$ be any maximal disjoint subfamily of $\left\{B(x,\frac{\rho(x))}{1+C^2}\right\}_{x\in\R^d}$ (of course, any maximal subfamily is countable). If $x\neq x_j$ for every $j\in\N$, then by maximality there exists a $k$ such that $B\left(x,\frac{\rho(x)}{1+C^2}\right)$ intersects $B\left(x_k,\frac{\rho(x_k)}{1+C^2}\right)$. Picking a point in the intersection and using twice \eqref{radius-ineq}, we see that $\rho(x)\leq C^2\rho(x_k)$, and thus that $x\in B(x_k,\rho(x_k))$. This proves (i).

To see that (ii) holds, fix $k$ and consider the indices $j_1,\dots,j_N$ corresponding to balls of the covering intersecting $B(x_k,\rho(x_k))$. By the same argument as above, we see that $C^{-2}\rho(x_k)\leq \rho(x_{j_\ell})\leq C^2\rho(x_k)$. This means that $B(x_k,(1+2C^2)\rho(x_k))$ contains $B\left(x_{j_\ell}, \frac{\rho(x_{j_\ell})}{1+C^2}\right)$ for every $\ell$. These balls are disjoint by construction and their radius is $\geq\frac{\rho(x_k)}{C^2(1+C^2)} $, therefore $N$ has to be bounded by a constant which depends only on $C$ and the dimension $d$.
\end{proof}

\subsection{Radius functions associated to potentials}\label{radiustopot-sec}

Now consider a measurable function\be
V:\R^d\rightarrow [0,+\infty).
\ee 
We assume that $V$ is locally bounded, not almost everywhere zero, and satisfies the following \emph{$L^\infty$-doubling condition}:\bel\label{pot-doubling}
||V||_{L^\infty(B(x,2r))}\leq D ||V||_{L^\infty(B(x,r))}\qquad\forall x\in\R^d,\ r>0,
\eel
where $D<+\infty$ is a constant independent of $x$ and $r>0$.

We want to associate to every such $V$ a certain radius function. Before giving the detailed arguments, let us describe the heuristics behind it. \newline

\par If we have a free quantum particle moving in $\R^d$ and $B$ is a ball of radius $r$, the uncertainty principle asserts that in order to localize the particle on the ball $B$ one needs an energy of the order of $r^{-2}$. If the particle is not free, but it is subject to a potential $V$, this energy increases by the size of $V$ on $B$. This means in particular that if $B$ is such that $\max_B V\leq r^{-2}$, then the amount of energy required for the localization is comparable to the one in the free case: in this case \emph{one does not feel the potential on $B$}. The radius function $\rho_V$ we are going to describe gives at every point $x$ the largest radius $\rho_V(x)$ such that one cannot feel the potential $V$ on $B(x,\rho_V(x))$.\newline

To formalize the discussion above, we begin by defining the function \be
f(x,r):=r^2||V||_{L^\infty(B(x,r))}\qquad (x\in\R^d, r>0).\ee 
We highlight two properties of $f$:\begin{enumerate}
\item $f(x,r)$ is strictly monotone in $r$ for every fixed $x$. 
\item $\lim_{r\rightarrow0+}f(x,r)=0$ and $\lim_{r\rightarrow+\infty}f(x,r)=+\infty$ for every $x$.
\end{enumerate}
To verify them, it is useful to observe that since $V$ is not almost everywhere $0$, an iterated application of \eqref{pot-doubling} shows that $||V||_{L^\infty(B)}>0$ for every non empty ball $B$.

We define \be
\rho_V(x):=\sup\{r>0:\ f(x,r)\leq 1\}.
\ee
By properties (1) and (2) above the $\sup$ exists and it is positive and finite. 

\begin{prop}\label{pot-rsquared}
We have\be
\frac{\rho_V(x)^{-2}}{4D}\leq||V||_{L^\infty(B(x,\rho_V(x)))} \leq \rho_V(x)^{-2}\qquad\forall x\in\R^d.
\ee
\end{prop}

\begin{proof}
The right inequality follows immediately from the definition of $\rho_V$. To prove the one on the left, observe that \eqref{pot-doubling} implies\bee
f(x, 2\rho_V(x))&=&4\rho_V(x)^2||V||_{L^\infty(B(x,2\rho_V(x)))}\\
&\leq& 4D\rho_V(x)^2||V||_{L^\infty(B(x,\rho_V(x)))}.
\eee The definition of $\rho_V(x)$ shows that the last term is smaller than $4D$, while the first one is larger than $1$. This finishes the proof.
\end{proof}

The next two results together prove that $\rho_V$ is a radius function.

\begin{prop} The function $\rho_V$ is Borel. 
\end{prop}

\begin{proof}
We have to see that $\{x:\rho_V(x)>t\}$ is a Borel set for every $t>0$, but \bee
\{x:\rho_V(x)>t\}&=&\{x:\ \exists r>t \text{ s.t. } f(x,r)\leq 1\}\\
&=&\{x:\ \exists r\in\Q\cap(t,+\infty) \text{ s.t. } f(x,r)\leq 1\}\\
&=&\cup_{r\in\Q\cap(t,+\infty)}\{x:\ ||V||_{L^\infty(B(x,r))}\leq r^{-2}\}.
\eee
It then suffices to verify that $||V||_{L^\infty(B(x,r))}$ is Borel in $x$ for every fixed $r>0$. In fact, $||V||_{L^\infty(B(\cdot,r))}$ is lower semi-continuous: $||V||_{L^\infty(B(x,r))}>u$ if and only if $V>u$ on a subset of positive measure of $B(x,r)$, and this property is clearly preserved by small perturbations of the center $x$.
\end{proof}

\begin{prop}\label{pot-to-rad} The function $\rho_V$ satisfies the following inequalities for every $x,y\in\R^d$: \be
 C^{-1}  \max\left\{\frac{|x-y|}{\rho_V(x)},1\right\}^{-M_1}\rho_V(x)\leq \rho_V(y)\leq C  \max\left\{\frac{|x-y|}{\rho_V(x)},1\right\}^{M_2}\rho_V(x),
\ee where $C, M_1, M_2$ depends only on the $L^\infty$-doubling constant $D$ appearing in \eqref{pot-doubling}. 

In particular $\rho_V$ is a radius function.
\end{prop}

\begin{proof} We have already seen that $\rho_V:\R^d\rightarrow(0,+\infty)$ is well-defined and Borel. Assume that $|x-y|<2^k \rho_V(x)$, for some integer $k\geq 1$. If $|x-y|<s<2^k \rho_V(x)$ we have\bee
s^2||V||_{L^\infty(B(y,s))}&\leq& s^2||V||_{L^\infty(B(x,2s))}\\
&\leq&D^{k+1}s^2||V||_{L^\infty(B(x,2^{-k}s))}\\
&=& 2^{2k}D^{k+1}(2^{-k}s)^2||V||_{L^\infty(B(x,2^{-k}s))}\\
&\leq&2^{2k}D^{k+1},
\eee
where in the third line we used $k+1$ times \eqref{pot-doubling} and in the last one we used the fact that $2^{-k}s<\rho_V(x)$.

 In particular $f(y,2^{-k}D^{-\frac{k+1}{2}}s)\leq 1$ and hence $2^{-k}D^{-\frac{k+1}{2}}s\leq \rho_V(y)$. By the arbitrariness of $s<2^k \rho_V(x)$, we conclude that \bel\label{pot-half-bound}
 \rho_V(x)\leq D^{\frac{k+1}{2}}\rho_V(y)\leq 2^{Mk}\rho_V(y),
 \eel for an integer $M$ depending only on $D$. 
 
 Inequality \eqref{pot-half-bound} gives $|x-y|<2^{(M+1)k}\rho_V(y)$, so that we can apply the above argument with $x$ and $y$ inverted, we conclude that $\rho_V(y)\leq 2^{M(M+1)k}\rho_V(x)$. Now the thesis follows choosing $k$ such that $2^k$ is comparable to $\max\left\{\frac{|x-y|}{\rho_V(x)},1\right\}$.\end{proof}

\section{Admissible weights}\label{adm-sec}

It is now time to introduce the class of weights to which our main results apply.

\begin{dfn}\label{adm-dfn} A $C^2$ plurisubharmonic weight $\varphi:\C^n\rightarrow\R$ is said to be admissible if:\begin{enumerate}
\item[\emph{(1)}] the following $L^\infty$ doubling condition holds:\be
\sup_{B(z,2r)}\Delta\varphi\leq D\sup_{B(z,r)}\Delta\varphi \quad\forall z\in \C^n,\ r>0,
\ee for some finite constant $D$ which is independent of $z$ and $r$,
\item[\emph{(2)}] there exists $c>0$ such that 
\bel\label{adm-lower}
\inf_{z\in\C^n}\sup_{w\in B(z,c)}\Delta\varphi(w)>0.
\eel
\end{enumerate}
\end{dfn}

If $\varphi$ is an admissible weight, then \be V\equiv\Delta\varphi:\C^n\rightarrow[0,+\infty)\ee satisfies condition \eqref{pot-doubling} of Section \ref{radius-sec} (we are identifying $\C^n$ and $\R^{2n}$), and is continuous and not everywhere zero, because of \eqref{adm-lower}. Thus we have the associated radius function $\rho_{\Delta\varphi}$ and distance function $d_{\Delta\varphi}$. Since here we are dealing only with one fixed weight $\varphi$, we can drop the subscript and denote them just by $\rho$ and $d$. We call $\rho$ the \emph{maximal eigenvalue radius function} and $d$ the \emph{maximal eigenvalue distance} corresponding to the weight $\varphi$. The reason for this name is simple: as we have already remarked, $\Delta\varphi$ is four times the trace of the complex Hessian of  $\varphi$ and hence it is comparable to its maximal eigenvalue.

\begin{prop}\label{adm-bounded}
The maximal eigenvalue radius function associated to an admissible weight is bounded.
\end{prop}

\begin{proof}
By the definition of the radius function associated to a potential (see Section \ref{radius-sec}), \be
\rho(z):=\sup\{r>0: \sup_{w\in B(z,r)}\Delta\varphi(w)\leq r^{-2}\},
\ee and the statement follows immediately from \eqref{adm-lower}.
\end{proof}

The next lemma will play a key role in later sections.

\begin{lem}\label{adm-hol} Let $\varphi$ and $\rho$ be as above. There exists a constant $C$ depending only on $\varphi$ such that if $h:B(z,r)\rightarrow\C$ is holomorphic and $r\leq \rho(z)$, then \be
|h(z)|^2e^{-2\varphi(z)}\leq \frac{C}{|B(z,r)|}\int_{B(z,r)}|h|^2e^{-2\varphi}.
\ee 
\end{lem}
Notice that the above estimate holds for every ball if $\varphi=0$ (it follows immediately from the mean-value property for $h$). One can think of $\rho(z)$ as the maximal scale at which \emph{one does not feel the weight}. This should be compared with the heuristic discussion in Section \ref{radiustopot-sec}. 

The proof of Lemma \ref{adm-hol} is based on the following proposition.

\begin{prop}\label{adm-gauge} Let $\varphi$ and $\rho$ be as above. For every $z\in\C^n$ there exists a $C^2$ function $\psi:B(z,\rho(z))\rightarrow\R$ such that $H_\varphi=H_\psi$, i.e., $\frac{\partial^2\varphi}{\partial z_j\partial \overline{z}_k}=\frac{\partial^2\psi}{\partial z_j\partial \overline{z}_k}$ $\forall j,k$, and such that \be
\sup_{w\in B(z,\rho(z))}|\psi(w)|\leq C_n,
\ee where $C_n$ is a constant which depends only on the dimension $n$.
\end{prop}

\begin{proof} We recall the following fact: if $\omega$ is a continuous and bounded $(1,1)$-form defined on $B(z,r)\subseteq\C^n$ such that:\begin{enumerate}
\item $\overline{\omega}=\omega$,   
\item it is $d$-closed in the sense of distributions,
\end{enumerate}
then there exists a real-valued, bounded and continuous function $\psi$ on $B(z,r)$ such that $i\partial\dbar\psi=\omega$ and $||\psi||_\infty\leq C_n r^2||\omega||_\infty$. The latter $L^\infty$ norm is the maximum of the $L^\infty$ norms of the coefficients of $\omega$. This is Lemma 4 of \cite{delin}, where a proof can be found. 

To deduce our proposition notice that $i\partial\dbar\varphi$, restricted to $B(z,\rho(z))$, satisfies conditions (1) and (2) above (recall that $\partial$ and $\dbar$ anti-commute).  The $L^\infty$ norm of $i\partial\dbar\varphi$ on $B(z,\rho(z))$ is bounded by a constant times $\rho(z)^{-2}$ by the definition of $\rho$ and the elementary observation that the coefficients of a non-negative matrix are bounded by its trace. Therefore there is a real-valued function $\psi$ on $B(z,\rho(z))$ such that $\partial\dbar\psi=\partial\dbar\varphi$ and $||\psi||_\infty\leq C_n$, as we wanted. Notice that $\psi-\varphi$ is harmonic, and hence smooth, so that $\psi$ has the same regularity as $\varphi$.
\end{proof}

\begin{proof}[Proof of Lemma \ref{adm-hol}]
By $A\lesssim B$ we mean $A\leq CB$, where $C$ is a constant depending only on $\varphi$. Fix $z$ and $r$ and let $\psi$ be the function given by Proposition \ref{adm-gauge}. Since $\psi-\varphi$ is pluriharmonic, there exists a holomorphic function $H$ on $B(z,r)$ such that $\Re(H)=\psi-\varphi$. If $h$ is as in the statement, using the $L^\infty$ bound on $\psi$, we can estimate
\be 
|h(z)|^2e^{-2\varphi(z)}\lesssim|h(z)|^2e^{2\psi(z)-2\varphi(z)}= |h(z)e^{H(z)}|^2.
\ee Applying the mean-value property and the Cauchy-Schwarz inequality to the holomorphic function $he^H$, we find
\bee
|h(z)e^{H(z)}|^2&\leq& \frac{1}{|B(z,r)|}\int_{B(z,r)}|he^H|^2\\
&=& \frac{1}{|B(z,r)|}\int_{B(z,r)}|h|^2 e^{\widetilde{\varphi}-\varphi}\\
&\lesssim& \frac{1}{|B(z,r)|}\int_{B(z,r)}|h|^2 e^{-\varphi}.
\eee This concludes the proof. \end{proof}

\section{Exponential decay of canonical solutions}\label{exp-sec}

Now that all the ingredients are in place, in this section we prove that if $\varphi$ is an admissible weight such that $\Box_\varphi$ is $\mu$-coercive and $\mu$ satisfies certain assumptions, then the canonical solutions of the weighted $\dbar$-problem exhibit a fast decay outside the support of the datum, in a way which is described in terms of $\mu$.

In the statement of the result we use the following terminology: a constant $C$ is \emph{allowable} if it depends only on $\varphi$, $\mu$ and the dimension $n$, and $A\lesssim B$ stands for the inequality $A\leq CB$, where the implicit constant $C$ is allowable. 

\begin{thm}\label{exp-thm} Let $\varphi$ be an admissible weight and assume that there exists \be\kappa:\C^n\rightarrow(0,+\infty)\ee such that:\begin{enumerate}
\item[\emph{(1)}] $\kappa$ is a bounded radius function, 
\item[\emph{(2)}] $\kappa(z)\geq\rho(z)$ for every $z\in\C^n$,
\item[\emph{(3)}] $\Box_\varphi$ is $\kappa^{-1}$-coercive.
\end{enumerate}
Recall that $\rho$ is the maximal eigenvalue function introduced in Section \ref{adm-sec}.

Then there are allowable constants $\eps, r_0,R_0>0$ such that the following holds true. Let $z\in\C^n$ and let $u\in L^2_{(0,1)}(\C^n,\varphi)$ be $\dbar$-closed and identically zero outside $B_\kappa(z,r_0)$. If $f$ is the canonical solution of \be
\dbar f=u,
\ee which exists by part \emph{(ii)} of Proposition \ref{coerc-prop}, then the pointwise bound\be
|f(w)|\lesssim e^{\varphi(w)}\kappa(z)e^{-\eps d_\kappa(z,w)} \rho(w)^{-n}||u||_\varphi
\ee
 holds for every $w$ such that $d_\kappa(z,w)\geq R_0$.
\end{thm}

A few comments before the proof:\begin{enumerate}
\item The distance $d_\kappa$ and the corresponding metric balls $B_\kappa(z,r)$ associated to $\kappa$ are defined in Section \ref{radius-sec}.
\item The definition of $\mu$-coercivity (Definition \ref{coerc-def}) shows that $\mu$ is dimensionally the inverse of a length, and this is consistent with our requirement that $\kappa=\mu^{-1}$ be a radius function.
\item If $\Box_\varphi$ is $\kappa^{-1}$-coercive for some bounded radius function $\kappa$ that does not satisfy condition (2) of the statement, then we can put $\widetilde{\kappa}:=\kappa\vee\rho$. By Proposition \ref{radius-max} of Section \ref{radius-sec}, $\widetilde{\kappa}$ is a radius function, condition (2) is trivially satisfied, and $\Box_\varphi$ is $\widetilde{\kappa}^{-1}$-coercive, because $\widetilde{\kappa}^{-1}\leq \kappa$.
\end{enumerate}

\begin{proof} By Proposition \ref{radius-prop} of Section \ref{radius-sec} we can find allowable constants $r_0\in(0,1)$ and $R_0\geq 2$ such that \bel\label{exp-inclusion}
B_\kappa(z,r_0)\subseteq B(z,\kappa(z)/2)\subseteq B(z,\kappa(z))\subseteq B_\kappa(z,R_0-1).
\eel
There are also allowable constants $r_1,r_2\in(0,1)$ such that \bel\label{exp-inclusion2} B(w,2r_1\kappa(w))\subseteq B_\kappa(w,r_2)\subseteq B(w,\kappa(w)).\eel 
If $d_\kappa(z,w)\geq R_0$, we have $B_\kappa(z,R_0-1)\cap B_\kappa(w,r_2)=\varnothing$. Since $\kappa(w)\geq\rho_{\text{max}}(w)$, the canonical solution $f$ is holomorphic on $B(w,r_1\rho(w))$. By Lemma \ref{adm-hol}, we have\be
|f(w)|^2e^{-2\varphi(w)}\lesssim \rho(w)^{-2n}\int_{B(w,r_1\rho(w))}|f|^2e^{-2\varphi}.
\ee 
Recall from part (ii) of Proposition \ref{coerc-prop} that $f=\dbar^*_\varphi N_\varphi u$. Since $\Box_\varphi N_\varphi u=u$ vanishes on $B(w,2r_1\kappa(w))$, Lemma \ref{MKH-caccioppoli} yields
\bee
\int_{B(w,r_1\rho(w))}|f|^2e^{-2\varphi}&\leq& \int_{B(w,r_1\kappa(w))}|f|^2e^{-2\varphi}\\
&\lesssim& \kappa(w)^{-2}\int_{B(w,2r_1\kappa(w))}|N_\varphi u|^2e^{-2\varphi}.
\eee
Putting our estimates together, we see that we are left with the task of proving the $L^2$ estimate (with $\eps>0$ admissible):\bel\label{exp-L2}
\kappa(w)^{-2}\int_{B(w,2r_1\kappa(w))}|N_\varphi u|^2e^{-2\varphi} \lesssim \kappa(z)^2e^{-2\eps d_\kappa(z,w)}||u||_\varphi^2.
\eel
Let $\ell:[0,+\infty)\rightarrow[0,1]$ be the continuous function equal to $0$ on $[0,\kappa(z)/2]$, equal to $1$ on $[\kappa(z),+\infty)$, and affine in between. By \eqref{exp-inclusion} and \eqref{exp-inclusion2}, the function $\eta(z'):=\ell(|z'-z|)$ is equal to $0$ on $B(w,2r_1\kappa(w))$, equal to $1$ on $B_\kappa(w,r_2)$, and 
\bel\label{exp-nabla-eta}\sup_{z'\in B(z,\kappa(z))}|\nabla\eta(z')|=\frac{\kappa(z)^{-1}}{2}.\eel 
We also need to define $b(z'):=\min\{d_\kappa(z,z'),d_\kappa(z,w)\}$. We know by Proposition \ref{radius-prop} that $d_\kappa(z,\cdot)$ is Lipschitz, and hence $b$ is also Lipschitz. Moreover, estimate \eqref{radius-nabla} gives \bel\label{exp-nabla-b}
|\nabla b(z')|\lesssim\kappa(z')^{-1},\eel  and $||b||_\infty\leq d_\kappa(z,w)$. From these facts, one may easily conclude that $\eta e^{\eps b}$ is a real-valued bounded Lipschitz function, for any $\eps>0$. By Proposition \ref{MKH-computation}, we obtain
\bee 
\mathcal{E}_\varphi(\eta e^{\eps b}N_\varphi u)&=&\frac{1}{4}\int_{\C^n}|\nabla(\eta e^{\eps b})|^2|N_\varphi u|^2e^{-2\varphi} + \Re(\eta e^{\eps b} u,\eta e^{\eps b} N_\varphi u)_\varphi \\
&\lesssim&\int_{\C^n}|\nabla\eta|^2 e^{2\eps b}|N_\varphi u|^2e^{-2\varphi}+\eps^2\int_{\C^n}\eta^2 e^{2\eps b}|\nabla b|^2|N_\varphi u|^2e^{-2\varphi},
\eee
where we used the fact that $u$ vanishes on the support of $\eta$. By the $\kappa^{-1}$-coercivity of $\Box_\varphi$, \eqref{exp-nabla-eta} and \eqref{exp-nabla-b}, we get\bee
\int_{\C^n}\kappa^{-2}\eta^2 e^{2\eps b}|N_\varphi u|^2e^{-2\varphi}&\lesssim&\kappa(z)^{-2}\int_{B(z,\kappa(z))} e^{2\eps b}|N_\varphi u|^2e^{-2\varphi}\\
&+&\eps^2\int_{\C^n}\eta^2 e^{2\eps b}\kappa^{-2}|N_\varphi u|^2e^{-2\varphi}.
\eee If $\eps\leq\eps_0$, where $\eps_0$ is allowable, recalling that on $B(z,\kappa(z))\subseteq B_\kappa(z,R_0-1)$ we have $e^{2\eps b(z')}\leq e^{2\eps d_\kappa(z,z')}\lesssim 1$, we find
\be\int_{\C^n}\kappa^{-2}\eta^2 e^{2\eps b}|N_\varphi u|^2e^{-2\varphi}\lesssim \kappa(z)^{-2}\int_{B(z,\kappa(z))}|N_\varphi u|^2e^{-2\varphi}.\ee
Notice that:\begin{enumerate}
\item[(a)] $b\geq d_\kappa(z,w)-1$ on $B(w, 2r_1\kappa(w))\subseteq B_\kappa(w,r_2)$, and that $\eta\equiv1$ on this ball,
\item[(b)] $\kappa^{-2}\gtrsim \kappa(w)^{-2}$ on $B(w,\kappa(w))$, and hence on $B(w, 2r_1\kappa(w))$, because $\kappa$ is a radius function,
\item[(c)] $\kappa(z)^{-2}\lesssim \kappa^{-2}$ on $B(z,\kappa(z))$, again because $\kappa$ is a radius function.
\end{enumerate}
For $\eps>0$ allowable, we then have
\bel\label{exp-1}
\kappa(w)^{-2}\int_{B(w, 2r_1\kappa(w))}|N_\varphi u|^2e^{-2\varphi}\lesssim e^{-2\eps d_\kappa(z,w)}\int_{\C^n}\kappa^{-2}|N_\varphi u|^2e^{-2\varphi}.
\eel
By part (i) of Proposition \ref{coerc-prop} and the fact that $u$ is supported where $\kappa\lesssim \kappa(z)$, we have\bel\label{exp-2}
\int_{\C^n}\kappa^{-2}|N_\varphi u|^2e^{-2\varphi}\lesssim \kappa(z)^2||u||_\varphi^2.
\eel
Putting \eqref{exp-1} and \eqref{exp-2} together we finally obtain \eqref{exp-L2} and hence the thesis.
\end{proof}

\section{Pointwise bounds for weighted Bergman kernels}\label{bergman-sec}
 
 To prove the pointwise bounds for weighted Bergman kernels we use a technique introduced in \cite{kerzman}, and adapted to the weighted case in \cite{delin}. For the sake of completeness, we state as a lemma the relevant part of \cite{delin} and recall its proof. 
 
 We continue working under the assumptions of Theorem \ref{exp-thm}, that is $\varphi$ is an admissible weight and $\kappa$ is a bounded radius function such that $\kappa\geq \rho$ and $\Box_\varphi$ is $\kappa^{-1}$-coercive.
 
 Let $\eta$ be a radial test function supported on the unit ball of $\C^n$ such that $\int_{\C^n}\eta=1$, and put \be\eta_z(w):=\frac{1}{(\delta\rho_{\text{max}}(z))^{2n}}\eta\left(\frac{w-z}{\delta\rho(z)}\right),\ee where $\delta>0$ is an allowable constant chosen so that the support of $\eta_z$, i.e., $B(z,\delta\rho(z))$, is contained in $B_\kappa(z,r_0)$, with $r_0$ as in Theorem \ref{exp-thm} (this is possible by Proposition \ref{radius-prop} of Section \ref{radius-sec}).

\begin{lem}\label{bergman-fz}
For every $z\in\C^n$ there exists a holomorphic function $H_z$ defined on $B(z,\rho(z))$ that vanishes in $z$ and such that\be
f_z:=\eta_z e^{\overline{H_z}+2\varphi}\in \mathcal{D}_0(\dbar).
\ee Moreover, we have the following inequalities\be
||f_z||_\varphi\lesssim e^{\varphi(z)}\rho(z)^{-n},\quad ||\dbar f_z||_\varphi\lesssim e^{\varphi(z)}\rho(z)^{-n-1}.
\ee
\end{lem}

\begin{proof} Let $\psi$ be the function given by Proposition \ref{adm-gauge} and $F$ the holomorphic function on $B(z,\rho(z))$ such that $\psi-\varphi=\Re(F)$. We define $H_z(w):=F(w)-F(z)$. Proposition \ref{adm-gauge} also gives the bound\bel\label{bergman-gauge}
||\varphi+\Re(F)||_\infty\lesssim1.
\eel
Let us check that $f_z:=\eta_z e^{\overline{H_z}+\varphi}$ verifies the inequalities of the statement. First of all,\bee
||f_z||_\varphi^2&\lesssim&\rho(z)^{-4n}\int_{B(z,\delta\rho(z))}|e^{H_z+2\varphi}|^2e^{-2\varphi}\\
&=&\rho(z)^{-4n}e^{-2\Re(F(z))}\int_{B(z,\delta\rho(z))}e^{2\Re(F)+2\varphi}\\
&\lesssim &\rho(z)^{-4n}e^{2\varphi(z)}\rho(z)^{2n}=\rho(z)^{-2n}e^{2\varphi(z)}, 
\eee where in the third line we used \eqref{bergman-gauge}. This proves the bound on $||f_z||_\varphi$. 

Next, we compute (using again \eqref{bergman-gauge})\bee
||\dbar f_z||_\varphi^2&=&\int_{\C^n}|\dbar f_z|^2e^{-2\varphi}\\
&=&e^{-2\Re(F(z))}\int_{\C^n}|\dbar \eta_z+\eta_z\dbar(\overline{F}+2\varphi)|^2e^{2\Re(F)+2\varphi}\\
&\lesssim& e^{-2\varphi(z)}\int_{\C^n}|\dbar \eta_z|^2+e^{-2\varphi(z)}\int_{\C^n}\eta_z^2|\dbar(2\Re(F)+2\varphi)|^2,
\eee where in the last term we used the fact that $\dbar F=0$. The key observation is that, since $\Delta(\Re(F)+\varphi)=\Delta\varphi\geq0$,\bee
\Delta(e^{2\Re(F)+2\varphi})e^{-2\Re(F)-2\varphi}&=&\Delta(2\Re(F)+2\varphi)+|\nabla(2\Re(F)+2\varphi)|^2\\
&\geq&4|\dbar(2\Re(F)+2\varphi)|^2.
\eee Coming back to our estimate, we have\bee
e^{-2\varphi(z)}\int_{\C^n}\eta_z^2|\dbar(2\Re(F)+2\varphi)|^2&\lesssim&e^{-2\varphi(z)}\int_{\C^n}\eta_z^2\Delta(e^{2\Re(F)+2\varphi})e^{-2\Re(F)-2\varphi}\\
&\lesssim &e^{-2\varphi(z)}\int_{\C^n}\Delta(\eta_z^2)e^{2\Re(F)+2\varphi}\lesssim e^{-2\varphi(z)}\int_{\C^n}\Delta(\eta_z^2), 
\eee where we used an integration by parts and \eqref{bergman-gauge}. Since it is easily seen that \be
\int_{\C^n}|\dbar \eta_z|^2+\int_{\C^n}\Delta(\eta_z^2)\lesssim \rho(z)^{-2n-2}, 
\ee the estimates of the statement are proved.\end{proof}

We can finally state our main result.

\begin{thm}\label{bergman-thm}
Under the assumptions of Theorem \ref{exp-thm}, there is an allowable constant $\eps>0$ such that the pointwise bound
\be
|K_\varphi(z,w)|\lesssim e^{\varphi(z)+\varphi(w)}\frac{\kappa(z)}{\rho(z)}\frac{e^{-\eps d_\kappa(z,w)}}{\rho(z)^{n}\rho(w)^{n}}
\ee 
holds for every $z,w\in\C^n$.
\end{thm}

\begin{proof}
For $z\in\C^n$ let $f_z$ be as in Lemma \ref{bergman-fz} and notice that \bee
B_\varphi(f_z)(w)&=&\int_{\C^n}K_\varphi(w,w')f_z(w')e^{-2\varphi(w')}d\mathcal{L}(w')\\
&=&\int_{\C^n}K_\varphi(w,w')\eta_z(w') e^{\overline{H_z}(w')}d\mathcal{L}(w')\\
&=& K_\varphi(w,z)e^{\overline{H_w}(w)}=\overline{K_\varphi(z,w)},
\eee where in the last line we used the fact that $\eta_w$ is radial with respect to $w\in\C^n$, $\int_{\C^n}\eta_w=1$, $K_\varphi(z,\cdot)e^{\overline{H_w}}$ is harmonic, being the product of two anti-holomorphic functions, and $H_w(w)=0$. Hence, by formula \eqref{coerc-bergman} of Proposition \ref{coerc-prop}, we have\bel\label{bergman-bar}
\overline{K_\varphi(z,w)}=f_z(w)-\dbar^*_\varphi N_\varphi \dbar f_z (w).
\eel
Since $\overline{K_\varphi(z,\cdot)}$ is holomorphic, Lemma \ref{adm-hol} yields \be
|K_\varphi(z,w)|\lesssim e^{\frac{\varphi(w)}{2}}\rho(w)^{-n}||K_\varphi(z,\cdot)||_\varphi.
\ee Thanks to \eqref{bergman-bar} and inequality \eqref{coerc-canonical-bound} of Proposition \ref{coerc-prop}, we have
\bee
||K_\varphi(z,\cdot)||_\varphi&\leq&||f_z||_\varphi+||\dbar^*_\varphi N_\varphi \dbar f_z||_\varphi\\
&\lesssim& ||f_z||_\varphi+||\kappa\dbar f_z||_\varphi\\ 
&\lesssim& ||f_z||_\varphi+\left(\max_{B(z,\rho(z))}\kappa\right)||\dbar f_z||_\varphi.
\eee Now recall that $\kappa\geq\rho$ and hence that $\kappa$, being a radius function, is $\lesssim \kappa(z)$ on $B(z,\rho(z))$. Lemma \ref{bergman-fz} finally gives\be
||K_\varphi(z,\cdot)||_\varphi\lesssim  \frac{\kappa(z)}{\rho(z)}e^{\varphi(z)}\rho(z)^{-n}.
\ee
What we obtained until now is \be
|K_\varphi(z,w)|\lesssim e^{\varphi(z)+\varphi(w)}\frac{\kappa(z)}{\rho(z)}\rho(z)^{-n}\rho(w)^{-n}.
\ee
This is equivalent to the conclusion of the theorem if $d_\kappa(z,w)\lesssim 1$. We can then assume from now on that $d_\kappa(z,w)\geq R_0$, with $R_0$ the allowable constant in Theorem \ref{exp-thm}, which then implies \be
|\dbar^*_\varphi N_\varphi \dbar f_z (w)|\lesssim e^{\varphi(w)}\kappa(z)e^{-\eps d_\kappa(z,w)} \rho(w)^{-n}||\dbar f_z||_\varphi.
\ee
We conclude by Lemma \ref{bergman-fz} and the identity $\overline{K_\varphi(z,w)}=-\dbar^*_\varphi N_\varphi \dbar f_z (w)$ (which holds for $d_\kappa(z,w)\geq R_0$).
\end{proof}

\section{Weighted Kohn Laplacians\\ and matrix Schr\"odinger operators}\label{kohnschrod-sec}

In this section we show that weighted Kohn Laplacians are unitarily equivalent to certain generalized Schr\"odinger operators. This will help to specialize Theorem \ref{bergman-thm} to the case in which the eigenvalues are comparable.

In order to do that we have first to define the generalization of Schr\"odinger operators we need. 

Let us introduce the following notation: if $X(\R^d)$ is a space of scalar-valued functions defined on $\R^d$, we denote by $X(\R^d,\C^m)$ the space of $m$-dimensional vectors whose components are elements of $X(\R^d)$.

\subsection{Matrix magnetic Schr\"odinger operators}

A \emph{matrix magnetic Schr\"odinger operator} acts on $\C^m$-valued functions defined on a Euclidean space $\R^d$, and is determined by the following two data:\begin{enumerate}
\item an $m\times m$ Hermitian matrix-valued \emph{electric potential} $V$,

\item a $C^1$ \emph{magnetic potential} $A:\R^d\rightarrow\R^d$.
\end{enumerate}

The formal expression of the magnetic matrix Schr\"odinger operator is the following:\be
\mathcal{H}_{V,A}\psi:=-\Delta_A\psi+V\psi\qquad\forall \psi\in C^2(\R^d,\C^m),
\ee where $\Delta_A =-\sum_{k=1}^d\left(-\frac{\partial}{\partial x_k}+iA_k\right)\left(\frac{\partial}{\partial x_k}-iA_k\right)$ acts diagonally, i.e., componentwise, on $\psi$, and $V$ acts by pointwise matrix multiplication. Explicitly, \be
\mathcal{H}_{V,A}\psi=\left(-\Delta_A\psi_k+\sum_{\ell=1}^mV_{k\ell}\psi_\ell \right)_{k=1}^m,
\ee
or in other words, $\mathcal{H}_{V,A}$ is the matrix differential operator:\be
\mathcal{H}_{V,A}=\begin{bmatrix}
-\Delta_A+V_{11} &\cdots& V_{1m}\\
\\
\vdots&\ddots&\vdots\\
\\
V_{m1} &\cdots& -\Delta_A+V_{mm}
\end{bmatrix}.
\ee

Since $V$ is Hermitian at every point and $\Delta_A$ is formally self-adjoint, $\mathcal{H}_{V,A}$ is formally self-adjoint too: \be
\int_{\R^d}(\mathcal{H}_{V,A}\psi,\phi)=\int_{\R^d}(\psi,\mathcal{H}_{V,A}\phi) \qquad\forall \psi,\phi\in C^2_c(\R^d,\C^m),
\ee where $(\cdot,\cdot)$ is the hermitian scalar product in $\C^m$. Notice that \bel\label{schrod-energy}
\mathcal{E}_{V,A}(\psi):=\int_{\R^d}(\mathcal{H}_{V,A}\psi,\psi)=\int_{\R^d}|\nabla_A\psi|^2+\int_{\R^d}(V\psi,\psi),
\eel
where \bel\label{mag-grad}|\nabla_A\psi|^2=\sum_{k=1}^m|\nabla_A\psi_k|^2.\eel 
The first term of the right-hand side of \eqref{schrod-energy} is called the \emph{kinetic energy}, while the second is the \emph{potential energy} of $\psi$. Notice that $(V\psi,\psi)$ is the pointwise evaluation of the quadratic form associated to $V$ on $\psi$. 

If $m=1$, $\mathcal{H}_{V,A}$ is the usual magnetic Schr\"odinger operator $-\Delta_A+V$ with scalar potential $V$, and the energy takes the form $\int_{\R^d}|\nabla_A\psi|^2+\int_{\R^d}V|\psi|^2$.

The case $m\geq2$ can not be reduced to the scalar one in general, unless the matrices $V(x)$ ($x\in\R^d$) can be simultaneously diagonalized.

Observe that the above discussion defines matrix magnetic Schr\"odinger operators only formally: we are not saying anything about the domains on which they are self-adjoint, as this will not be needed for our purposes.

Matrix Schr\"odinger operators without a magnetic potential attracted some attention in the recent mathematical physics literature (see, e.g., \cite{frank-lieb-seiringer}).

\subsection{Kohn Laplacians and matrix magnetic Schr\"odinger operators}
Let $\varphi:\C^n\rightarrow \R$ be $C^2$ and plurisubharmonic. We identify $\C^n$ with $\R^{2n}$ using the real coordinates $(x_1,y_1,\dots,x_n,y_n)$ such that $z_j=x_j+iy_j$ for every $j$. It will be useful to define the \emph{symplectic gradient of $\varphi$}:\bel\label{kohnschrod-perp}
\nabla^\perp\varphi:=\left(-\frac{\partial\varphi}{\partial y_1},\frac{\partial\varphi}{\partial x_1},\dots,-\frac{\partial\varphi}{\partial y_n},\frac{\partial\varphi}{\partial x_n}\right).
\eel

It is easy to verify that the mapping\bee
U_\varphi: L^2(\C^n,\C^n)&\longrightarrow &L^2_{(0,1)}(\C^n,\varphi)\\
 u=(u_1,\dots,u_n)&\longmapsto& \sum_{j=1}^n e^\varphi u_jd\overline{z}_j
\eee
is a surjective unitary transformation. If $u\in C^2_c(\C^n,\C^n)$ then $U_\varphi u\in\mathcal{D}(\Box_\varphi)$, because it is a $(0,1)$-form with $C^2_c$ coefficients.

\begin{prop}\label{kohnschrod-prop}
Consider the $n\times n$ Hermitian matrix-valued electric potential \be
V=8H_\varphi-4\text{tr}(H_\varphi)I_n,
\ee where $I_n$ is the $n\times n$ identity matrix, and the magnetic potential \be
A=\nabla^\perp\varphi.
\ee
We have the following identity\bel\label{kohnschrod-id}
\Box_\varphi (U_\varphi u)=U_\varphi\left(\frac{1}{4}\mathcal{H}_{A,V}u\right)\qquad \forall u\in C^2_c(\C^n,\C^n).
\eel
\end{prop}


Recall that while $\Box_\varphi$ is a genuine self-adjoint operator, the matrix Schr\"odinger operator $\mathcal{H}_{A,V}$ has been defined only formally. Identity \eqref{kohnschrod-id} may be used to extend $\mathcal{H}_{V,A}$ to a domain on which it is self-adjoint and unitarily equivalent to the weighted Kohn Laplacian.

The proof of Proposition \ref{kohnschrod-prop} is based on a computation, which we present as a separate lemma in order to be able to use it again later.

\begin{lem}\label{kohnschrod-comp}
If $A$ is as in Proposition \ref{kohnschrod-prop}, we have\be
\sum_{j=1}^n\int_{\C^n}\left|\frac{\partial (e^\varphi f)}{\partial \overline{z}_j} \right|^2e^{-2\varphi}=\frac{1}{4}\int_{\C^n}|\nabla_A f|^2-\frac{1}{4}\int_{\C^n}\Delta\varphi|f|^2\qquad\forall f\in C^2_c(\C^n).
\ee
\end{lem}

\begin{proof}

Define the vector fields \be
X_j:=\frac{\partial}{\partial x_j}+i\frac{\partial\varphi}{\partial y_j},\quad Y_j:=\frac{\partial}{\partial y_j}-i\frac{\partial\varphi}{\partial x_j}\qquad (j=1,\dots,n).
\ee  Recalling \eqref{mag-grad}, we have $|\nabla_Af|^2=\sum_{j=1}^n\left(|X_jf|^2+|Y_jf|^2\right)$. Notice that\be
\frac{\partial}{\partial \overline{z}_j}+\frac{\partial\varphi}{\partial \overline{z}_j}=\frac{1}{2}\left(X_j+iY_j\right).
\ee and that the formal adjoints of $X_j$ and $Y_j$ are $-X_j$ and $-Y_j$ respectively. Therefore we have \bee
&&\int_{\C^n}\left|\frac{\partial (e^\varphi f)}{\partial \overline{z}_j} \right|^2e^{-2\varphi}=\sum_{j=1}^n\int_{\C^n}\left|\frac{\partial f}{\partial \overline{z}_j}+\frac{\partial\varphi}{\partial \overline{z}_j}f\right|^2\\
&=&\frac{1}{4}\sum_{j=1}^n\int_{\C^n}\left|(X_j+iY_j)f\right|^2\\
&=&\frac{1}{4}\sum_{j=1}^n\left(\int_{\C^n}\left|X_jf\right|^2+\int_{\C^n}\left|X_jf\right|^2+\int_{\C^n}X_jf\overline{iY_j f}+\int_{\C^n}iY_jf\overline{X_j f}\right)\\
&=&\frac{1}{4}\sum_{j=1}^n\left(\int_{\C^n}|\nabla_Af|^2-i\int_{\C^n}[X_j,Y_j]f\cdot \overline{f}\right)
\eee
Observing that $\sum_{j=1}^n[X_j,Y_j]=-i\Delta\varphi$, we obtain the thesis.
\end{proof}

\begin{proof}[Proof of Proposition \ref{kohnschrod-prop}]

Since both $\Box_\varphi$ and $\mathcal{H}_{V,A}$ are formally self-adjoint, by polarization it is enough to prove that \be
(\Box_\varphi U_\varphi u, U_\varphi u )_\varphi=(U_\varphi\mathcal{H}_{V,A}u,U_\varphi u)_\varphi=(\mathcal{H}_{V,A}u,u)_0,
\ee where the parenthesis on the right represent the scalar product in the unweighted space $L^2(\C^n,\C^n)$. Using Morrey-Kohn-H\"ormander formula and Lemma \ref{kohnschrod-comp}, we can write \bee
&&(\Box_\varphi U_\varphi u, U_\varphi u )_\varphi=\mathcal{E}_\varphi(U_\varphi u)\\
&=&\sum_{j,k=1}^n\int_{\C^n}\left|\frac{\partial (e^\varphi f)}{\partial \overline{z}_j} \right|^2e^{-2\varphi}+2\int_{\C^n}(H_\varphi u,u)\\
&=&\sum_{k=1}^n\left(\frac{1}{4}\int_{\C^n}|\nabla_A u_k|^2-\frac{1}{4}\int_{\C^n}\Delta\varphi|u_k|^2\right)+2\int_{\C^n}(H_\varphi u,u)\\
&=&\frac{1}{4}\left(\int_{\C^n}|\nabla_A u|^2+\int_{\C^n}((8H_\varphi-\Delta\varphi I_n) u,u)\right).
\eee To complete the proof notice that \be
\text{tr}(H_\varphi)=\sum_{j=1}^n\frac{\partial^2\varphi}{\partial z_j\partial \overline{z}_j}=\frac{1}{4}\Delta\varphi.
\ee and recall \eqref{schrod-energy}.\end{proof}

\subsection{One complex variable versus several complex variables in the analysis of $\Box_\varphi$}\label{onevsmany-sec}

Proposition \ref{kohnschrod-prop} reveals a radical difference between the one-dimensional case ($n=1$) and the higher-dimensional case ($n\geq2$) in the analysis of the weighted Kohn Laplacian. If $n=1$, the potential $V$ is the scalar function \be
8H_\varphi-4\text{tr}(H_\varphi)=4\text{tr}(H_\varphi)=\Delta\varphi,
\ee which is non-negative, while if $n\geq2$ the potential $V$ is matrix-valued and \be
\text{tr}(V)=\text{tr}(8H_\varphi-4\text{tr}(H_\varphi)I_n)=(8-4n)\text{tr}(H_\varphi)=(2-n)\Delta\varphi
\ee is non-positive. As a consequence, the potential $V$ always has non-positive eigenvalues if $n\geq2$. 

In the one-variable case one may combine Proposition \ref{kohnschrod-prop}, identity \eqref{schrod-energy} and the diamagnetic inequality \bel\label{diamag}
|\nabla_A u|\geq |\nabla|u||,
\eel which holds almost everywhere for $u\in C^1$ (see \cite{christ} for a proof). The result is \bee
\mathcal{E}_\varphi(U_\varphi u)&=&\frac{1}{4}\mathcal{E}_{V,A}(u)\\
&=&\frac{1}{4}\left(\int_\C|\nabla_Au|^2+\int_\C V|u|^2\right)\\
&\geq&\frac{1}{4}\left(\int_\C|\nabla|u||^2+\int_\C\Delta\varphi|u|^2\right).
\eee

The last term is the energy $\mathcal{E}_{\Delta\varphi,0}(|u|)$ of the compactly supported Lipschitz function $|u|$ in presence of the scalar electric potential $\Delta\varphi$ and of no magnetic field. If one can prove the bound \bel\label{onevsmany-bound}
\mathcal{E}_{\Delta\varphi,0}(u)\geq \int_\C \mu^2 |u|^2,
\eel for some $\mu:\C\rightarrow[0,+\infty)$ and for all Lipschitz functions $u$, one can immediately deduce that \be
\mathcal{E}_\varphi(u)\geq  \int_\C \left(\frac{\mu}{2}\right)^2 |u|^2e^{-2\varphi}, 
\ee i.e., that $\Box_\varphi$ is $\frac{\mu}{2}$-coercive. Notice that \eqref{onevsmany-bound} is a Fefferman-Phong inequality (like the one treated in the Appendix). This is the approach followed by Christ in \cite{christ}.\newline

Such a route is not viable in general in several variables: if $u\in C^\infty_c(\C^n,\C^n)$, applying the diamagnetic inequality we get:\bee
\mathcal{E}_\varphi(U_\varphi u)&=&\frac{1}{4}\mathcal{E}_{V,A}(u)\\
&=&\frac{1}{4}\left(\sum_{k=1}^n\int_{\C^n}|\nabla_Au_k|^2+\int_{\C^n} (Vu,u)\right)\\
&\geq&\frac{1}{4}\sum_{k=1}^n\left(\int_{\C^n}|\nabla|u_k||^2+\int_{\C^n}\lambda(V)|u_k|^2\right),
\eee
 where $\lambda(V)$, the minimal eigenvalue of $V$, is everywhere non-positive. The last term is $\sum_{k=1}^n\mathcal{H}_{\lambda(V),0}(|u_k|)$, which is not even non-negative in general, so no estimate like \eqref{onevsmany-bound} can hold.

Nevertheless, a variant of this approach can be very useful in the special case in which the eigenvalues of $H_\varphi$ are comparable, as we show in the next section.

\section{Pointwise estimates of weighted Bergman kernels\\ when the eigenvalues are comparable}\label{comp-sec}

\subsection{Statement of the result}

\begin{thm}\label{bergman-comp}
Let $\varphi:\C^n\rightarrow\R$ be $C^2$, plurisubharmonic and such that:
\begin{enumerate}
\item[\emph{(1)}] there exists $c>0$ such that 
\be
\inf_{z\in\C^n}\sup_{w\in B(z,c)}\Delta\varphi(w)>0,
\ee
\item[\emph{(2)}] $\Delta\varphi$ satisfies the reverse-H\"older inequality
\bel\label{bergman-RH}
||\Delta\varphi||_{L^\infty(B(z,r))}\leq Ar^{-2n}\int_{B(z,r)}\Delta\varphi\qquad\forall z\in\C^n, r>0,
\eel for some $A<+\infty$,
\item[\emph{(3)}] the eigenvalues of $H_\varphi$ are comparable, i.e., there exists $\delta>0$ such that \bel\label{mucomparable-ineq}
(H_\varphi(z)v,v)\geq \delta \Delta\varphi(z)|v|^2\qquad\forall z\in\C^n, \ v\in\C^n.
\eel 
\end{enumerate}
Then there is an allowable constant $\eps>0$ such that the pointwise bound
\be
|K_\varphi(z,w)|\lesssim e^{\varphi(z)+\varphi(w)}\frac{e^{-\eps d(z,w)}}{\rho(z)^{n}\rho(w)^{n}}
\ee 
holds for every $z,w\in\C^n$, where $d$ is the maximal eigenvalue distance associated to $\varphi$ (see Section \ref{adm-sec}).
\end{thm}

Since $\Delta\varphi/4$ is the trace of $H_\varphi$, condition (3) is clearly equivalent to the global comparability of any pair of eigenvalues of $H_\varphi$. If \eqref{mucomparable-ineq} holds, necessarily $4\delta\leq1$.

In order to prove Theorem \ref{bergman-comp} we discuss $\mu$-coercivity under the assumption of comparability of eigenvalues.

\subsection{$\mu$-coercivity when the eigenvalues are comparable}

\begin{lem}\label{mucomparable-lem}
Let $\varphi:\C^n\rightarrow\R$ be $C^2$, plurisubharmonic and such that \begin{enumerate}
\item[\emph{(1)}] the eigenvalues of $H_\varphi$ are comparable, i.e., \eqref{mucomparable-ineq} holds,
\item[\emph{(2)}] $\Delta\varphi$ satisfies the reverse-H\"older inequality
\bel\label{mucomparable-RH}
||\Delta\varphi||_{L^\infty(B(z,r))}\leq Ar^{-2n}\int_{B(z,r)}\Delta\varphi\qquad\forall z\in\C^n, r>0.
\eel 
\end{enumerate}
Then $\Box_\varphi$ is $\mu$-coercive, where \be
\mu=c\rho^{-1}.
\ee Here $c>0$ is a constant which depends only on $\Delta\varphi$ and $\delta$, and $\rho$ is the radius function associated to the potential $\Delta\varphi$.
\end{lem}

To prove Lemma \ref{mucomparable-lem}, we  are going to use an argument which appears, e.g., in the proof of Theorem 5.6 of \cite{haslinger-helffer}, and a version of an inequality going back to Fefferman and Phong (see, e.g., \cite{fe-unc} or \cite{shen}), which we state first.

\begin{lem}\label{fph-lem}
Let $V:\R^d\rightarrow [0, +\infty)$ be a locally bounded function satisfying the reverse H\"older inequality\bel\label{fph-RH}
||V||_{L^\infty(B(x,r))}\leq Ar^{-d}\int_{B(x,r)}V\qquad\forall x\in\R^d, r>0,
\eel where $A<+\infty$ is a constant which is independent of $x$ and $r$.

There is a constant $C$ which depends only on $V$ such that for every $f\in C^1_c(\R^d)$ we have\be
\int_{\R^d}\rho_V^{-2}|f|^2\leq C\left(\int_{\R^d}|\nabla f|^2+\int_{\R^d}V|f|^2\right).
\ee
\end{lem}

Observe that \eqref{fph-RH} implies that $V(x)dx$ is a doubling measure, and hence $V$ satisfies \eqref{pot-doubling} (see, e.g., \cite{stein-bigbook}). In particular the radius function $\rho_V$ is well-defined. For the sake of completeness, a short proof of Lemma \ref{fph-lem} is given in the Appendix.

\begin{proof}[Proof of Lemma \ref{mucomparable-lem}]
Let $u\in C^2_c(\C^n,\C^n)$. We have \bee
\mathcal{E}_\varphi(U_\varphi u)&=&\sum_{j,k}\int_{\C^n}\left|\frac{\partial (e^\varphi u_j)}{\partial\overline{z}_k}\right|^2e^{-2\varphi}+2\int_{\C^n} (H_\varphi u,u)\\
&\geq&\sum_{j=1}^n\left(4\delta \sum_{k=1}^n\int_{\C^n}\left|\frac{\partial (e^\varphi u_j)}{\partial\overline{z}_k}\right|^2e^{-2\varphi}+2\delta\int_{\C^n}\Delta\varphi|u_j|^2\right),
\eee 
where we used the Morrey-Kohn-H\"ormander formula, hypothesis (1), and the inequality $4\delta\leq 1$. We invoke Lemma \ref{kohnschrod-comp} to obtain\be
\mathcal{E}_\varphi(U_\varphi u)\geq \delta\sum_{j=1}^n\left(\int_{\C^n}|\nabla_A u_j|^2+\int_{\C^n}\Delta\varphi|u_j|^2\right).
\ee
We can now apply the diamagnetic inequality \eqref{diamag} to deduce that\be
\mathcal{E}_\varphi(U_\varphi u)\geq  \delta\sum_{j=1}^n\left(\int_{\C^n}|\nabla |u_j||^2+\int_{\C^n}\Delta\varphi|u_j|^2\right).
\ee 
Since we assumed hypothesis (2), we can apply the Fefferman-Phong inequality (Lemma \ref{fph-lem}):\be
\mathcal{E}_\varphi(U_\varphi u)\geq C^{-1} \delta\int_{\C^n}\rho^{-2}|u|^2.
\ee We now replace $u$ with $e^{-\varphi} u$ and conclude by an elementary approximation argument (to remove the restriction that $u$ be $C^2_c$) that we omit.
\end{proof}

\subsection{Proof of Theorem \ref{bergman-comp}}
The reverse-H\"older inequality \eqref{bergman-RH} implies that the measure with density $\Delta\varphi$ with respect to Lebesgue measure is doubling, i.e.,\be
\int_{B(z,2r)}\Delta\varphi\lesssim\int_{B(z,r)}\Delta\varphi\qquad\forall z\in\C^n,\ r>0.
\ee
This, together with the reverse-H\"older inequality itself, implies condition (1) in Definition \ref{adm-dfn}. Since condition (2) of that definition is among our hypotheses, the weight $\varphi$ is admissible.

By Lemma \ref{mucomparable-lem}, $\Box_\varphi$ is $c\rho^{-1}$-coercive, where $c>0$ is admissible. An application of Theorem \ref{bergman-thm} with $k=c^{-1}\rho$ gives the thesis.

\appendix

\section{Proof of the Fefferman-Phong inequality (Lemma \ref{fph-lem})}\label{fph-sec}

The function $f$ is fixed throughout the proof. If $x\in\R^d$, we put $B=B(x,\rho_V(x))$. Integrating in $(y,y')\in B\times B$ the trivial bound\be
V(y)|f(y')|^2\leq 2V(y)|f(y)-f(y')|^2+2V(y')|f(y')|^2, 
\ee we get\be
\int_BV \int_B|f|^2\leq2 ||V||_{L^\infty(B)}\int_B\int_B|f(y)-f(y')|^2dydy'+\omega_d\rho_V(x)^d\int_B V|f|^2,
\ee 
where $\omega_d$ is the measure of the unit ball of $\R^d$.
We recall that we have the following form of Poincar\'e inequality: \be
\int_{B\times B}|f(y)-f(y')|^2dydy'\leq C_d r^{d+2}\int_B|\nabla f|^2,
\ee where $C_d$ is a constant depending only on $d$ and $B$ is any euclidean ball of radius $r$. Combining it with Proposition \ref{pot-rsquared} we find\bel\label{fph-quad}
\int_BV \int_B|f|^2\leq2C_d \rho_V(x)^d\int_B|\nabla f|^2+\omega_d\rho_V(x)^d\int_B V|f|^2,
\eel  
The reverse H\"older condition and Proposition \ref{pot-rsquared} give\bel\label{fph-int}
\int_BV\geq \frac{1}{A}\rho_V(x)^d||V||_{L^\infty(B)}\geq \frac{1}{4DA}\rho_V(x)^{d-2}.
\eel
Putting \eqref{fph-quad} and \eqref{fph-int} together, and using Proposition \ref{pot-to-rad} to bring $\rho_V^{-2}$ inside the integral, we obtain\bel\label{fph-quad-2}
\int_B\rho_V^{-2}|f|^2\leq C'\left( \int_B|\nabla f|^2+\int_B V|f|^2\right),
\eel where $C'$ depends on $V$, but not on $x$ or $f$. Summing the inequalities \eqref{fph-quad-2} corresponding to the points $x_j$ given by Proposition \ref{radius-covering} (applied to $\rho_V$) we obtain\be
\int_{\R^d}\rho_V^{-2}|f|^2\leq C'K\left( \int_{\R^d}|\nabla f|^2+\int_{\R^d} V|f|^2\right),
\ee where $K$ is the constant appearing in Proposition \ref{radius-covering}. Putting $C=C'K$ we obtain the statement.

\bibliographystyle{amsalpha}
\bibliography{weighted-bergman}

\end{document}